\documentclass[a4paper,12pt,
]{article}

\usepackage[margin=3cm,footskip=1cm]{geometry}

\usepackage[T1]{fontenc}

\usepackage{amsmath,amssymb,amsthm,bm,mathtools,xspace,booktabs,url,mathstyle}
\usepackage{graphicx,etoolbox}
\usepackage[breaklinks]{hyperref}
\usepackage[shortlabels]{enumitem}
\setlist{
  listparindent=\parindent,
  parsep=0pt,
}

\usepackage[round]{natbib}

\AtBeginEnvironment{thebibliography}{%
  \small% smaller font size
  \setlength\itemsep{0.75em plus 0.5em minus 0.75em}%
}

\usepackage{caption}
\usepackage[raggedright]{titlesec}

\titleformat\paragraph{\itshape}{hang}{}{}{}
\titlespacing*\paragraph{0em}{*2}{*1.5}

\numberwithin{equation}{section}

\theoremstyle{plain} % the default, bold header, italic body
\newtheorem{theorem}{Theorem}[section]
\newtheorem{lemma}[theorem]{Lemma}
\newtheorem{proposition}[theorem]{Proposition}
\newtheorem{definition}[theorem]{Definition}

\theoremstyle{definition} % bold header, upright body
\newtheorem{example}[theorem]{Example}
\newtheorem{remark}[theorem]{Remark}
\newtheorem{condition}[theorem]{Condition}
\theoremstyle{plain}

\usepackage{authblk}

\makeatletter
\newcommand\CorrespondingAuthor[1]{%
  \begingroup%
  \def\@makefnmark{}%
  \footnotetext{Corresponding author: #1}%
  \endgroup%
}
\makeatother

\makeatletter
\renewenvironment{abstract}{%
  \small%
  \providecommand\keywords{%
    \par\medskip\noindent\textit{Keywords:}\xspace}%
  \providecommand\amssc{%
    \par\medskip\noindent\textit{AMS Subject Classification:}\xspace}%
  \begin{center}%
    \bfseries \abstractname\vspace{-.5em}\vspace{\z@}%
  \end{center}%
  \quote%
}{\endquote}
\makeatother

\usepackage[above,below,section]{placeins}

\usepackage[load-configurations=abbreviations]{siunitx}

\usepackage[all]{onlyamsmath}

\mathtoolsset{centercolon} % := properly vertically aligned
\usepackage{bbm} % blackboard math that works with \mathbbm{1}

\makeatletter
\newlength\dlf
\newsavebox\hest
\newsavebox\hestii
\def\ddanglefactor{0.5}

\newcommand\dda@measure[1]{
  \global\sbox\hestii{$\m@th\currentmathstyle
    \left\langle\vphantom{#1}\right.$}
}

\newcommand\dda@autoscale[1]{
  \dda@measure{#1}
  \left\langle\kern-\ddanglefactor\wd\hestii\left\langle
      #1
    \right\rangle\kern-\ddanglefactor\wd\hestii\right\rangle
}

\newcommand\dda@manualscale[2][\@gobble]{
  \@nameuse{\expandafter\@gobble\string #1l}\langle
  \kern-\ddanglefactor\wd\hestii
  \@nameuse{\expandafter\@gobble\string #1l}\langle
  #2
  \@nameuse{\expandafter\@gobble\string #1r}\rangle
  \kern-\ddanglefactor\wd\hestii
  \@nameuse{\expandafter\@gobble\string #1r}\rangle
}

\newcommand\ddangle{\@ifstar{\dda@autoscale}{\dda@manualscale}}

\makeatother

\newcommand{\cA}{\mathcal{A}}
\newcommand{\cB}{\mathcal{B}}
\newcommand{\cC}{\mathcal{C}}

\newcommand{\cF}{\mathcal{F}}

\newcommand{\cS}{\mathcal{S}}

\newcommand{\cR}{\mathcal{R}}
\newcommand{\cL}{\mathcal{L}}
\newcommand{\cX}{\mathcal{X}}

\newcommand{\bP}{\mathbb{P}}
\newcommand{\bN}{\mathbb{N}}
\newcommand{\bR}{\mathbb{R}}

\newcommand{\bE}{\mathbb{E}}
\newcommand{\bL}{\mathbb{L}}
\newcommand{\bM}{\mathbb{M}}

\newcommand{\rC}{\mathrm{C}}

\DeclarePairedDelimiter\norm{\lVert}{\rVert}
\DeclarePairedDelimiter\abs{\lvert}{\rvert}

\DeclareMathOperator*{\plim}{\mathbb{P}-lim}

\newcommand{\lss}{\text{LSS}\xspace}
\newcommand{\bss}{\text{BSS}\xspace}
\newcommand{\ID}{\text{ID}\xspace}
\newcommand{\SD}{\text{SD}\xspace}
\newcommand{\nullarg}{{}\cdot{}}
\newcommand{\indf}{\mathbbm{1}}

\newcommand\olspan{\operatorname{\overline{span}}}

\begin{document}
\title{Selfdecomposable Fields}

\author[1]{Ole E. Barndorff-Nielsen} \author[2]{Orimar Sauri}
\author[1]{Benedykt Szozda}

\affil[1]{Department of Mathematics, Aarhus University,
  \texttt{\{oebn,szozda\}@math.au.dk}} \affil[2]{Department of Economics,
  Aarhus University, \texttt{osauri@creates.au.dk}}

\date{}

\maketitle

\begin{abstract}
  In the present paper we study selfdecomposability of random fields,
  as defined directly rather than in terms of finite-dimensional
  distributions.  The main tools in our analysis are the master L\'evy
  measure and the associated L\'evy-It\^o representation.  We give the
  dilation criterion for selfdecomposability analogous to the
  classical one.  Next, we give necessary and sufficient conditions
  (in terms of the kernel functions) for a Volterra field driven by a
  L\'evy basis to be selfdecomposable.  In this context we also study
  the so-called Urbanik classes of random fields.  We follow this with
  the study of existence and selfdecomposability of integrated
  Volterra fields.  Finally, we introduce infinitely divisible
  field-valued L\'evy processes, give the L\'evy-It\^o representation
  associated with them and study stochastic integration with respect
  to such processes.  We provide examples in the form of L\'evy
  semistationary processes with a Gamma kernel and Ornstein-Uhlenbeck
  processes.


  \amssc 60E07, 60G51, 60G60
\end{abstract}

\section{Introduction}

\label{sec:introduction}

Of the many interesting classes of infinitely divisible distributions
(cf.\ for example~\cite{Bon92}) that of selfdecomposable laws -- the
\SD class -- has a foremost position. Originally this class was
defined as the family of limit distributions of normalized partial
sums. Paul L\'{e}vy was the first to study this family in depth. In
particular he determined the form of, what is now known as the
L\'{e}vy measure, of a selfdecomposable distribution. In the early
literature the class was also referred to as L\'{e}vy's probability
measures. For a long time these laws held a rather anonymous
position. Even in Michel Lo\`{e}ve's detailed and beautifully written
biographical account of L\'{e}vy's life and contributions to
Probability Theory (\cite{Loeve73}) the concept of selfdecomposability
is not mentioned.  And in volume II of Feller's ``An Introduction to
Probability Theory and Its Applications'' it is treated only very
briefly (Section~XVII.8), as a 'special topic' under the name of class
$L$. L\'{e}vy himself described his work on selfdecomposability in his
monumental monograph \emph{Theorie de l'addition des variables
  al\'{e}atoires} (1937).

The more recent prominence of selfdecomposability came from the
realisation, due to~\cite{Wolfe82}, that any \SD distribution can be
represented as that of a stochastic integral with respect to a L\'evy
process, the integrand being the negative exponential; or, otherwise
put, the limit law of the solution to a linear stochastic differential
equation driven by a L\'evy process $L$ is selfdecomposable provided
the L\'evy measure of the L\'{e}vy seed of $L$ satisfies a log moment
condition; and all selfdecomposable laws are representable in this
way. In turn this gave rise to the concept of L\'evy-driven OU
processes, continuous time Markov processes whose marginals are
\SD. For an account of the developments in regard to stochastic
integral representations of classes of \ID laws in the period
1982--2010 see~\cite{Jur11}.

Looking at this from a modelling point of view, suppose that subject
matter knowledge and empirical data indicate that the phenomenon under
study might be described as a continuous time stationary process. The
simplest type of such a process is a Markov process and for a model to
be consistent with this the one-dimensional marginal of the process
must be \SD, an assumption that may be supported by the available
knowledge.

In general, a stochastic process has traditionally been said to be \SD
if all its finite-dimensional distributions are \SD. However,
in~\cite{BNMS06a} it was proposed to define selfdecomposability of a
stochastic process directly, saying that a process
$X=(X_{t})_{t\in \bR}$ is \SD if for all $q\in(0,1)$ it can be
represented in law as the sum of $qX$ plus an independent stochastic
process $V^{(q)}$. It is this approach we take in the present paper
where we more generally study selfdecomposable stochastic fields.

In Mathematical Finance, in Turbulence and in other fields, OU
processes, and the extended concept of supOU processes, have had an
important role as models for stochastic volatility, see for
instance~\cite{BNStel13}. We intend here to develop the similar
approach for stochastic fields, with the aim of incorporating such
fields in models of Ambit Stochastics type, in particular as regards
applications to turbulence studies. Note that in Turbulence, and many
other fields of Physics, stochastic volatility is referred to as
intermittency.

A central object in this context is the master L\'{e}vy measure of a
stochastic field, a concept introduced in~\cite{Mar70}, see
also~\cite{BNMS06a}, and recently brought on a more analytically
tractable footing by~\cite{Ros07a,Ros07b,Ros08,Ros13}. Importantly,
there Rosi\'{n}ski also discusses an associated L\'{e}vy-It\^{o}
representation. In the following we will build substantially on the
results and propositions presented
in~\cite{Ros07a,Ros07b,Ros08,Ros13}.

With the master L\'{e}vy measure and the associated L\'{e}vy-It\^{o}
representation in hand it is in particular possible to characterize
volatility/intermittency fields generated from \SD fields in much the
same way that OU processes are engendered from \SD random variables.

The present paper is organized as
follows. Section~\ref{sec:background} provides background material on
ambit fields, Volterra fields, L\'{e}vy bases and integration with
respect to L\'{e}vy bases. In Section~\ref{sec:some-levy-theory},
based on the recent work of~\cite{Ros07a,Ros07b,Ros08,Ros13}, we
introduce the concept of the master L\'{e}vy measure of \ID fields and
present the associated L\'{e}vy-It\^{o} representation. Finally, we
give the dilation criterion for selfdecomposability of \ID
fields. Section~\ref{section4} is devoted to the study of
selfdecomposability of Volterra fields. In particular, we study the
master L\'{e}vy measure of Volterra fields and give conditions on the
kernel of a Volterra field that ensure inheritance of the \SD property
of the background driving noise to the resulting Volterra field. We
close Section~\ref{section4} with the converse result, that is we give
conditions under which the Volterra field is \SD if and only if the
background driving noise is \SD. All of the results mentioned above
hold if we exchange \SD class with the Urbanik class $\bL_{m}$. In
Section~\ref{sec:integr-id-volt} we study the existence and
selfdecomposability of integrated Volterra
fields. Section~\ref{sec:id-field-valued} is devoted to \ID
field-valued L\'{e}vy processes. We give a L\'{e}vy-It\^{o}
representation of such processes and study integration with respect to
such processes. We close the section with the study of Volterra and OU
type field-valued processes and their selfdecomposability. Section 7
concludes.

\section{Background}
\label{sec:background}
In the present section we give the definition of ambit fields and we
recall basic results related to L\'{e}vy bases and stochastic
integration with respect to L\'{e}vy bases.  Throughout this paper
$(\Omega ,\cF,\bP)$ denotes a complete probability space

\subsection{Ambit fields, Volterra fields and \lss processes}

Ambit fields are random fields describing the dynamics in a
stochastically developing field, for instance a turbulent wind
field. A key characteristic of the modelling framework of ambit
stochastics, which distinguishes this from other approaches is that
beyond the most basic kind of random input it also specifically
incorporates additional, often drastically changing, inputs referred
to as volatility or intermittency. Another distinguishing feature is
the presence of ambit sets that delineate which part of space-time may
influence the value of the field at any given point in space-time.

In terms of mathematical formulae, in its original form an ambit field
is specified by
\begin{align*}
  Y(t,x)={}& \mu +\int_{A(t,x)}g(t,s,x,\xi)\sigma(s,\xi)\,L(dsd\xi) \\
           & +\int_{D(t,x)}q(t,s,x,\xi)\chi(s,\xi)\,dsd\xi
\end{align*}
where $t$ denotes time while $x$ gives the position in $d$-dimensional
Euclidean space. Further, $A(t,x)$ and $D(t,x)$ are subsets of
$\bR \times \bR^{d}$, termed ambit sets, $g$ and $q$ are deterministic
weight functions, $\sigma$ and $\chi$ are stochastic fields
representing the volatility or intermittency. Finally, $L$ denotes a
L\'{e}vy basis (i.e.\ an independently scattered and infinitely
divisible random measure). For aspects of the theory and applications
of Ambit processes and fields
see~\cite{BNBV14a,BNBV11,BNBS14,BNLSV13,BNESch05,BNSch07,ChKl13,HedSch14b,Pod14}
and~\cite{Pakk14}.

A \textit{L\'{e}vy semistationary process} (\lss) is a stochastic
process $(Y_{t})_{t\in \bR}$ on a filtered probability space
$(\Omega ,\cF,(\cF_t)_{t\in \bR},\bP)$ which is described by the
following dynamics
\begin{equation*}
  Y_{t}
  = \theta
  + \int_{-\infty}^{t}g(t-s)\sigma_{s}\,dL_{s}
  + \int_{-\infty}^{t}q(t-s)a_{s}\,ds,
  \qquad t\in \bR,
\end{equation*}
where $\theta \in \bR$, $L$ is a L\'{e}vy process with triplet
$(\gamma ,b,\nu)$, $g$ and $q$ are deterministic functions such that
$g(x)=q(x)=0$ for $x\leq 0$, and $\sigma$ and $a$ are adapted
c\`{a}dl\`{a}g processes. When $L$ is a two-sided Brownian motion $Y$
is called \textit{Brownian semistationary process} (\bss). Observe
that an \lss process is a null-space Ambit field. For further
references to theory and applications of L\'{e}vy semistationary
processes, see for instance~\cite{VV12,BEV14,BFK13}.

\subsection{L\'{e}vy bases}

Denoting by $\ID(\bR^{n})$ the space of infinitely divisible (\ID for
short) distributions on~$\bR^{n}$, we recall that any
$\mu \in \ID(\bR^{n})$ has a L\'{e}vy-Khintchine representation given
by
\begin{equation*}
  \log \widehat{\mu}(\theta)
  = i\langle \theta ,\gamma \rangle
  - \tfrac{1}{2} \langle \theta ,B\theta \rangle
  + \int_{\bR^{n}}\left[ e^{i\langle
      \theta ,x\rangle}-1-i\langle \tau_{n}(x) ,\theta \rangle\right]
  \nu(dx),
  \qquad \theta \in \bR^{n},
\end{equation*}
where $\widehat{\mu }$ is the characteristic function of the law of
$\mu$, $\gamma \in \bR^{n}$, $B$ is a symmetric non-negative definite
matrix on $\bR^{n\times n}$ and $\nu$ is a L\'{e}vy measure, i.e.\
$\nu(\{ 0^{n}\}) =0$, with $0^{n}$ denoting the origin in $\bR^{n},$
and $\int_{\bR^{n}}1\wedge \abs{x}^{2}\nu(dx) <\infty .$ Here, we
assume that the truncation function $\tau_{n}$ is given by
$\tau_{n}(x_{1},\ldots ,x_{n}) = (\frac{x_{i}}{1\vee \abs{x_{i}}})_{i=1}^{n},\ (x_{1},\ldots ,x_{n}) \in \bR^{n}$. By
$\SD(\bR^{n}) ,$ we mean the subset of $\ID(\bR^{n})$ of
selfdecomposable (\SD) distributions on $\bR^{n}$. More precisely,
$\mu \in \ID( \bR^{n})$ belongs to $\SD(\bR^{n})$ if and only if for
any $q>1$ there exists $\mu_{q}\in \ID(\bR^{n})$ such that
\begin{equation*}
  \widehat{\mu }(\theta)
  = \widehat{\mu }(q^{-1}\theta) \widehat{\mu }_{q}(\theta)
  \qquad  \text{for any $\theta \in \bR^{n}$}.
\end{equation*}

Let $\cS$ be a non-empty set and $\cR$ a $\delta $-ring of subsets of
$\cS$ having the property that exists an increasing sequence
$\{ S_{n}\} \subset \cS$ with $\bigcup_{n}S_{n}=\cS$. A real-valued
stochastic field $L=\{ L(A) :A\in \cR\}$ defined on
$(\Omega ,\cF,\bP)$ is called \emph{independently scattered random
  measure} (i.s.r.m. for short), if for every sequence
$\{ A_{n}\}_{n\geq 1}$ of disjoint sets in $\cR$, the random variables
$(L(A_{n}))_{n\geq 1}$ are independent, and if
$\bigcup_{n\geq 1}A_{n}$ belongs to $\cR$, then we also have
\begin{equation*}
  L\Big(\bigcup_{n\geq 1}A_{n}\Big)
  = \sum_{n\geq 1}L(A_{n})
  \qquad\text{a.s.},
\end{equation*}
where the series is assumed to converge almost surely. When the law of
$L(A)$ belongs to $\ID(\bR)$ for any $A\in \cR$, $L$ is called a
\emph{L\'{e}vy basis}. Any L\'{e}vy basis admits a L\'{e}vy-Khintchine
representation:
\begin{equation*}
  \rC\{\theta \ddagger L(A) \}
  = \int_{A}\psi(\theta ,s) c(ds) ,
  \qquad \theta \in \bR,A\in \cR,
\end{equation*}
where $\rC\{\theta \ddagger X\}$ denotes the cumulant function of a
random variable $X$ and
\begin{equation}
  \psi(\theta ,s)
  \coloneqq \gamma(s) \theta
  - \tfrac{1}{2} b^{2}(s) \theta^{2}
  + \int_{\bR}[ e^{i\theta x}-1-i\theta \tau_{1}(x) ]
  \rho(s,dx) ,
  \qquad \theta \in \bR,s\in \cS.
  \label{eqn1.1}
\end{equation}
The functions $\gamma ,b$ and $\rho(\nullarg ,dx)$ are measurable with
$b\geq 0$ and $\rho(s,\nullarg)$ is a L\'{e}vy measure for every
$s\in \cS$. The measure $c$ is defined on
$\cB_{\cS}\coloneqq \sigma(\cR)$ and is called the \emph{control
  measure} of $L$.  We will refer to
$(\gamma(s) ,b(s) ,\rho(s, dx) ,c(d s))$ as the \emph{characteristic
  quadruplet} of $L$. If $L$ has characteristic quadruplet,
$(\gamma(s) ,b(s) ,\rho(s, dx) ,c(ds)) ,$ the associated family of
random variables $(L^{\prime }(s))_{s\in \cS}$ such that
$L^{\prime }(s)$ is \ID and has characteristic triplet
$(\gamma(s) ,b(s) ,\rho(s, dx))$ is called \emph{L\'{e}vy seeds.} When
$b\equiv 0$, we say that $L$ is \emph{Poissonian}. If $\gamma ,b$ and
$\rho$ do not depend on $s$ we say that $L$ is
\emph{factorizable}. Moreover, $L$ will be called \emph{homogeneous}
if it is factorizable and $c$ is proportional to the Lebesgue measure.

If we put $\cR=\cB_{b}(\bR^{k})$ the bounded Borel sets and add the
extra condition $L(\{ x\}) =0$ a.s. for all $x\in \bR^{k}$, $L$ has a
L\'{e}vy-It\^{o} decomposition: We have that almost surely
\begin{equation*}
  L(A)
  = \int_{A}\gamma(s) c(ds)
  + W(A)
  + \int_{A}\int_{\abs{x} >1} x N(dxds)
  + \int_{A}\int_{\abs{x} \leq 1} x \widetilde{N}(dxds),
  \quad A\in \cR,
\end{equation*}
where $W$ is a centered Gaussian L\'{e}vy basis with
$\bE(W( A) W(B)) =\int_{A\cap B}b(s) c(ds)$ for all $A,B\in \cR$,
$\widetilde{N}$ and $N$ are compensated and non-compensated Poisson
random measures on $\bR^{k}\times \bR$ with intensity
$\rho(s, dx) c(ds)$, respectively. Additionally, $W$ and $N$ are
independent. See~\cite{Ped03} for more details.

\subsection{Stochastic integration with respect to L\'{e}vy bases}
\label{stochintsec}

In the following, we present a short review of~\cite{RajRos89}
concerning to the existence of stochastic integrals of the form
$\int_{\cS}f(s) L(ds)$, where $f\colon\cS\rightarrow\bR$ is a
measurable function and $L$ a L\'{e}vy basis with characteristic
quadruplet $(\gamma(s) ,b( s) ,\rho(s, dx) ,c(ds))$.

Let $\cL^{0}(\Omega ,\cF,\bP)$ be the space of real-valued random
variables endowed with convergence in probability.  Consider
$\vartheta$, the space of simple functions on $(\cS,\cR)$, i.e.\
$f\in \vartheta$ if and only if $f$ can be written as
\begin{equation*}
  f(s)
  = \sum\limits_{i=1}^{k}a_{i}\indf_{A_{i}}(s),
  \qquad s\in \cS,
\end{equation*}
where $A_{i}\in \cR$ and $a_{i}\in \bR$ for $i=1,\ldots ,k$.  Given
$f\in \vartheta$, define the linear operator
$m\colon\vartheta \rightarrow \cL^{0}(\Omega ,\cF,\bP)$ by
\begin{equation}
  m(f)
  \coloneqq \sum\limits_{i=1}^{k}a_{i}L(A_{i}).
  \label{eqn1.3}
\end{equation}
In stochastic integration theory, commonly one is looking for a linear
extension of operators of the form \eqref{eqn1.3} to a suitable space,
let's say $I_{m}$, such that $m(f)$ can be approximated by simple
integrals of elements of $\vartheta$. More precisely, if $m$ can be
extended to $I_{m}$ and $\vartheta$ is dense in this set, we say that
$f$ is $L$-integrable or $f\in I_{m}$ and we define its stochastic
integral with respect to $L$ as
\begin{equation}
  \int_{\cS}f(s) L(ds)
  \coloneqq \plim_{n\rightarrow \infty} m(f_{n}),
  \label{eqn1.3.1}
\end{equation}
provided that $f_{n}\in \vartheta$ and $f_{n}\rightarrow f$ $c$-a.e.

\newpage

In~\cite{RajRos89}, it has been shown that the simple integral
\eqref{eqn1.3} can be extended to the so-called Musielak-Orlicz space:
\begin{equation*}
  I_{m}
  = \bigl\{ f\colon(\cS,\cB_{\cS})\rightarrow(\bR,\cB(\bR))
  : \int_{\cS}\Phi_{0}(|f(s)|,s)c(ds)<\infty\bigr\} ,
\end{equation*}
where
\begin{equation}
  \Phi_{p}(r,s)
  \coloneqq \sup_{\abs{c} \leq 1} H(cr,s)
  + b^{2}(s)r^{2}
  + \int_{\bR}
  [ |xr|^{p}\indf_{\{|xr|>1\}}+|xr|^{2}\indf_{\{|xr|\leq 1\}}]
  \rho(s,dx),
  \label{eqn1.3.2}
\end{equation}
with $p\geq 0,r\in \bR,s\in \cS$ and
\begin{equation}
  H(r,s)
  \coloneqq \abs[\big]{
    \gamma(s)r
    + \int_{\bR}[\tau_{1}(xr)-r\tau_{1}(x)]\rho(s,dx)},
  \qquad r\in \bR,s\in \cS.
  \label{eqn1.3.3}
\end{equation}
For a comprehensive introduction to Musielak-Orlicz spaces, we refer
to \cite{RaoRen94}.

When $f\in I_{m},$ $\int_{\cS}f(s)L(ds)$ is \ID and
\begin{equation}
  \rC\left\{ \theta \ddagger \int_{\cS}f(s)L(ds)\right\}
  =\int_{\cS}\psi(f(s)\theta ,s)c(ds),
  \qquad \theta \in \bR,
  \label{eqn0.1}
\end{equation}
with $\psi$ as in \eqref{eqn1.1}.

Fix $p\geq 0$ such that $\bE(\abs{L(A)}^{p}) <\infty$ for all
$A\in \cR$ and define
\begin{equation}
  \cL_{\Phi_{p}}
  \coloneqq \bigl\{ f\colon(\cS,\cB_{\cS})\rightarrow(\bR,\cB(\bR))
  :\int_{\cS}\Phi_{p}(\abs{f(s)},s) c(ds) <\infty \bigr\} .
  \label{eqn0.2}
\end{equation}
$\cL_{\Phi_{p}}$ is the space of $L$-integrable functions having
finite $p$-moment. When $p=0$, $\cL_{\Phi_{0}}=I_{m}$, i.e.\
$\cL_{\Phi_{0}}$ is the space of $L$-integrable functions.
Furthermore, $\cL_{\Phi_{p}}$ endowed with the Luxemburg norm
\begin{equation}
  \norm{f}_{\Phi_{p}}
  \coloneqq \inf \bigl\{a>0:\int_{\cS}\Phi_{p}(a^{-1}\abs{f(s)},s)
  c(ds) \leq 1\bigr\} ,
  \label{eqn0.3}
\end{equation}
is a separable Banach space. Observe that $f\in \cL_{\Phi_{p}}$ if and
only if $\norm{f}_{\Phi_{p}}<\infty $.

Recall that $\cL^{0}(\Omega ,\cF,\bP)$ is the space of random
variables endowed with the convergence in probability. The following
properties of $\int_{\cS}f(s) L( ds)$ will be useful for the rest of
the paper, see~\cite{RajRos89} for proofs:

\begin{enumerate}
\item The mapping
  $(f\in \cL_{\Phi_{p}}) \longmapsto \big( \int_{\cS}f(s) L(ds) \in \cL^{p}(\Omega ,\cF,\bP) \big)$
  is continuous, i.e.\ if $f_{n}\rightarrow 0$ in $\cL_{\Phi_{p}}$,
  then $\int_{\cS}f_{n}(s) L(ds) \rightarrow 0$ in
  $ \cL^{p}(\Omega ,\cF,\bP)$;

\item If $L$ is symmetric or centered, then for any $p\geq 0$ the
  mapping
  $(f\in \cL_{\Phi_{p}}) \longmapsto \big( \int_{\cS}f(s) L(ds) \in \cL^{p}(\Omega , \cF,\bP) \big)$
  is an isomorphism between $\cL_{\Phi_{p}}$ and
  $\cL^{p}(\Omega ,\cF,\bP)$.
\end{enumerate}

\section{Some L\'{e}vy theory of \ID fields}

\label{sec:some-levy-theory} In this part we introduce the notions of
infinite divisible and selfdecomposable fields as well as some basic
properties of such fields.

\subsection{Infinite divisibility and selfdecomposability of
  stochastic fields}

Let $U$ be a non-empty index set and $X=(X_{u})_{u\in U}$ be a
real-valued stochastic field defined on $(\Omega ,\cF,\bP)$. We say
that $X$ is \emph{infinitely divisible}, writing
$\cL(X)\in \ID(\bR^{U})$, with $\cL(X)$ denoting the law of the field
$X$, if for any $n\in \bN$ there are $X^{i,n}=(X_{u}^{i,n})_{u\in U}$,
$i=1,\ldots ,n$, independent and identically distributed stochastic
fields such that
\begin{equation*}
  X\overset{d}{=}X^{1,n}+X^{2,n}+\cdots +X^{n,n}.
\end{equation*}
In the same way, we say that $X$ is \emph{selfdecomposable}, writing
$\cL(X)\in \SD(\bR^{U})$, if for any $q>1$ there exists
$X^{^{\prime }}$, an independent copy of $X$, and $V^{q}$ a random
field independent of $X^{\prime }$, such that
\begin{equation*}
  X\overset{d}{=}q^{-1}X^{^{\prime }}+V^{(q)}.
\end{equation*}
Observe that when $U$ is finite, the definition of infinite
divisibility and selfdecomposability coincide with the usual concepts
of \ID and \SD random vectors. We denote by $\bL_{m}(\bR^{U})$, for
$m=0,1,\ldots$, the $m$-th \emph{Urbanik class}, i.e.\
$\cL(X)\in \bL_{m}(\bR^{U})$ if and only if $X$ is \SD and the field
$\cL(V^{(q)})\in \bL_{m-1}(\bR^{U})$. Here
$\bL_{0}(\bR^{U})=\SD(\bR^{U})$.

\subsection{The master L\'{e}vy measure of an \ID field}

In this subsection we extend the family of L\'{e}vy measures
associated to an \ID field to a measure in the space of paths, what we
refer to as the master L\'{e}vy measure. Such a measure was originally
introduced in \cite{Mar70}. Later on, it was studied in depth by
\cite{Ros07a,Ros07b,Ros08,Ros13}, who also established the associated
L\'{e}vy-It\^{o} representation. For completeness of our discussion of
this result, in the appendix of this paper we present a detailed
proof.

Let us introduce some notation. For a given non-empty set $U$, denote
by $\widehat{U}$ the collection of all finite subsets of $U$. For any
$\widehat{u}\in \widehat{U},$ we write
$\bR^{\widehat{u}}\coloneqq \Pi_{u\in \widehat{u}}\bR$, i.e.\
$\bR^{\widehat{u}}$ is $\#\widehat{u}$-dimensional Euclidean space
with $\#\widehat{u}$ denoting the cardinality of
$\widehat{u}$. Furthermore, $0^{\widehat{u}}$ denotes the origin in
$\bR^{\widehat{u}}$ and
$X_{\widehat{u}}\coloneqq \pi_{\widehat{u}}(X)=(X_{u})_{u\in \widehat{u}}$. Here
$\pi_{\widehat{u}}\colon\bR^{U}\rightarrow \bR^{\widehat{u}}$ is the
natural projection of $\bR^{U}$ into $\bR^{\widehat{u}}.$ For any
$\widehat{u},\widehat{v}\in \widehat{U},$ with
$\widehat{u}\subset \widehat{v}$, $\pi_{\widehat{v}\widehat{u}}$
denotes the natural projection of $\bR^{\widehat{v}}$ into
$\bR^{\widehat{u}}$.

As an extension of~\cite[Theorems 3.4, 3.6 and 3.7]{BNMS06a}, we have
that $\cL(X) \in \ID(\bR^{U})$ if and only if
$\cL(X_{ \widehat{u} }) \in \ID(\bR^{\widehat{u}})$ for any
$\widehat{u} \in \widehat{U}$. An analogous statement applies for
selfdecomposability and for the Urbanik classes. Moreover, the field
$X$ has associated a consistent system of characteristic triplets in
the sense of the following proposition:

\begin{proposition}
  \label{propsystemchtrip}
  Let $X=(X_{u})_{u\in U}$ be an \ID field. For any
  $\widehat{u}\in \widehat{U},$ let
  $(\gamma_{\widehat{u}},B_{\widehat{u}},\nu_{\widehat{u}})$ be the
  characteristic triplet of $\cL(X_{\widehat{u} })$.  Then, there are
  unique functions $B\colon U\times U\rightarrow \bR$ and
  $\Gamma \colon U\rightarrow \bR,$ such that
  $\gamma_{\widehat{u}}=\pi_{\widehat{u}}(\Gamma)$ and
  $B_{\widehat{u}}=(B(u,v))_{u,v\in \widehat{u}}$. In addition, we
  have
  \begin{equation}
    \nu_{\widehat{u}}
    =\nu_{\widehat{v}}\circ\pi_{\widehat{v}\widehat{u}}^{-1},
    \qquad \text{on $\cB(\bR^{\widehat{u}}\setminus 0^{\widehat{u}})$,
      for any $\widehat{u},\widehat{v}\in \widehat{U}$
      and $\widehat{u}\subset \widehat{v}$}.
    \label{eqn1.0}
  \end{equation}
  Reciprocally, given functions $B$ and $\Gamma$ as before and a
  collection of L\'{e}vy measures satisfying \eqref{eqn1.0}, there
  exists a unique (in law) field $X$ having characteristic triplets
  $(\gamma_{\widehat{u}},B_{\widehat{u}},\nu_{\widehat{u}})$.
\end{proposition}

\begin{remark}
  \label{remarropsystemchtrip}
  Observe that \eqref{eqn1.0} only holds on
  $\cB(\bR^{\widehat{u}}\setminus 0^{\widehat{u}}) $ .  Indeed, since
  in general
  $\nu_{\widehat{v}}\circ \pi_{\widehat{v} \widehat{u}}^{-1}(\{ 0^{\widehat{u}}\}) \neq 0$,
  $\nu_{\widehat{v}}\circ \pi_{\widehat{v}\widehat{u}}^{-1}$ could
  have an atom in the origin of $\bR^{\widehat{u}}$, consequently
  $\nu_{\widehat{u}}$ and
  $\nu_{\widehat{v}}\circ \pi_{\widehat{v}\widehat{u} }^{-1}$ coincide
  only outside of a neighborhood of zero.
\end{remark}

From Proposition~\ref{propsystemchtrip} and
Remark~\ref{remarropsystemchtrip} , we have that
$(\nu_{\widehat{u}})_{\widehat{u}\in \widehat{U}}$ does not form a
projective system of measures, so in general it is not possible to
extend $(\nu_{\widehat{u}})_{\widehat{u}\in \widehat{U}}$ to a unique
measure on $\cB(\bR)^{U},$ the cylindrical $\sigma $-algebra of $U$,
by standard arguments. But even when it is possible, such measure
could not in general be $\sigma$-finite, mainly because
$\{0^{U}\}\notin \cB(\bR)^{U}$ when $U$ is uncountable. This was
already pointed out in~\cite{Ros07a,Ros07b,Ros08,Ros13}

From now on, we will assume that $U$ is uncountable. The countable
case is well known. In view of the pointed out before, we introduce
the concept of a measure that does not charge zero.

\begin{definition}
  Let $U$ be an arbitrary index set. A measure $\nu$ on
  $\cB(\bR)^{U}$, the cylindrical $\sigma$-algebra of $\bR^{U},$
  \emph{does not charge zero} if there exists $U_{0}\subset U$
  countable, such that
  \begin{equation}
    \nu(\pi_{U_{0}}^{-1}(0^{U_{0}}))=0.
    \label{eqn1.3.6}
  \end{equation}
\end{definition}

With all the notation above, we are now ready to present one of the
main results that is going to be fundamental for the rest of the
paper:

\begin{theorem} [Rosi\'{n}ski 2013]
  \label{thmmasterlevymeasure}
  Let $X=(X_{u})_{u\in U}$ be an \ID field with
  $(\gamma_{\widehat{u}},B_{\widehat{u}},\nu_{\widehat{u}})$ being its
  system of characteristic triplets. Then there are
  $B\colon U\times U\rightarrow \bR$ and
  $\Gamma \colon U\rightarrow \bR$ unique functions, such that
  $\gamma_{\widehat{u}}=\pi_{\widehat{u}}(\Gamma)$ and
  $B_{\widehat{u}}=(B(u,v))_{u,v\in \widehat{u}}$. Additionally there
  is a measure on $(\bR^{U},\cB(\bR)^{U})$ such that
  \begin{equation}
    \nu_{\widehat{u}}(A)=\nu \circ \pi_{\widehat{u}}^{-1}(A),
    \qquad A\in \cB(\bR^{\widehat{u}}\setminus 0^{\widehat{u}}),
    \widehat{u}\in \widehat{U},
    \label{eqn1.3.4}
  \end{equation}
  and
  \begin{equation}
    \int_{\bR^{U}}1\wedge |\pi_{u}(x)|^{2}\nu(dx)<\infty ,\qquad
    u\in U.  \label{eqn1.3.7}
  \end{equation}
  If $\nu$ does not charge zero, then $\nu$ is the unique measure that
  doesn't charge zero for which \eqref{eqn1.3.4} holds.
\end{theorem}

A proof of this theorem is presented in the appendix.

\begin{remark}
  In general, the measure $\nu$ in Theorem~\ref{thmmasterlevymeasure}
  is not unique. It is mainly because $\nu$ may not be
  $\sigma $-finite. This issue has been already pointed out
  in~\cite{Ros07a,Ros07b,Ros08,Ros13}. At this point the concept of a
  measure in the cylindrical $\sigma$-field that does not charge zero
  plays an important role. However, the uniqueness can be obtained
  without condition \eqref{eqn1.3.6}.
\end{remark}

From the preceding proposition, for a given measure $\nu$ satisfying
\eqref{eqn1.3.4} and \eqref{eqn1.3.7} we can construct a consistent
system of L\'{e}vy measures by putting
$\nu_{\widehat{u}}({} \cdot{}) =\nu [ \pi_{\widehat{u}}^{-1}({}\cdot\nobreak {}\nobreak \setminus\nobreak 0^{\widehat{u}}) ] .$
If in addition, we consider functions
$B\colon U\times U\rightarrow \bR$ and
$\Gamma \colon U\rightarrow \bR$, such that for any
$\widehat{ u}\in \smash{\widehat{U}}$,
$\Gamma_{\widehat{u}}\coloneqq \pi_{\widehat{u}}^{-1}( \Gamma) \in \bR^{\widehat{u}}$
and $B_{\widehat{u}}=(B(u,v) )_{u,v\in \widehat{u}}$ is non-negative
definite, then there is a unique (in law) \ID field $X$ having
characteristic triplets
$(\Gamma_{\widehat{u}},B_{\widehat{u}},\nu_{ \widehat{u}}) .$ This
remark induces naturally the following definition:

\begin{definition}
  Let
  $( \Gamma_{\widehat{u}},B_{\widehat{u}},\nu_{\widehat{u}})_{\widehat{u}\in \widehat{U}}$
  be the system of characteristic triplets associated to the \ID field
  $X=(X_{u})_{u\in U}$. A measure $\nu$ on $\cB(\bR)^{U},$ the
  cylindrical $\sigma$-algebra of~$\bR^{U},$ is said to be \emph{a
    pseudo master L\'{e}vy measure} of $X$ if \eqref{eqn1.3.4} and
  \eqref{eqn1.3.7} hold. If in addition, $\nu$ does not charge zero,
  we say that $\nu$ is \emph{the master L\'{e}vy measure} of $X. $ In
  this case, we refer to $(\Gamma ,B,\nu)$ as the \emph{characteristic
    triplet} of $X$. When we write that $X$ has characteristic triplet
  $(\Gamma ,B,\nu)$, we are going to assume that $\nu$ does not charge
  zero.
\end{definition}

The following L\'{e}vy-It\^{o} representation follows easily from
Theorem~\ref{thmmasterlevymeasure} and, like
Theorem~\ref{thmmasterlevymeasure}, it was introduced
in~\cite{Ros07a,Ros07b,Ros08,Ros13}.

\begin{proposition}[Rosi\'{n}ski 2013]
  Let $X=(X_{u})_{u\in U}$ be an infinitely divisible field with
  characteristic triplet $(\Gamma ,B,\nu)$. Then the field
  \begin{equation*}
    \widetilde{X}_{u}\coloneqq \pi_{u}(\Gamma)+W_{u}+\int_{\bR^{U}}\pi_{u}(x)[ N(dx)-\indf_{\{|\pi_{u}(x)|\leq 1\}}\nu(dx)],\qquad u\in U,
  \end{equation*}
  is well defined and it is a version of $X$. Here $W$ is a centered
  Gaussian process with
  $Cov(X_{\widehat{u}})=(B(u,v))_{u,v\in \widehat{u}}$, for any
  $\widehat{u}\in \widehat{U}$. Further, $N$ is a Poisson random
  measure with intensity $\nu$ and it is independent of $W$.
\end{proposition}

\subsection{Criterion for selfdecomposability of \ID fields}

As in the classical theory, the selfdecomposability of an \ID field
can be characterized via dilations.

\begin{proposition}
  \label{propdilatationsd}
  Let $X=(X_{u})_{u\in U}$ be an infinitely divisible field with
  characteristic triplet $(\Gamma ,B,\nu) $. Then $X$ is
  selfdecomposable if and only if for any $q>1$
  \begin{equation}
    \nu(qA) \leq \nu(A) ,
    \qquad A\in \cB(\bR)^{U}.
    \label{eqn2}
  \end{equation}
\end{proposition}

\begin{proof}
  Let
  $( \Gamma_{\widehat{u}},B_{\widehat{u}},\nu_{\widehat{u}})_{\widehat{u}\in \widehat{U}}$
  be the system of characteristic triplets associated to $X$.  If
  \eqref{eqn2} holds, then from \eqref{eqn1.3.4}, we have that the
  same expression holds for $\nu_{\widehat{u}}$, which means that
  $\cL(X_{\widehat{u}}) \in \ID(\bR^{ \widehat{u}})$ for any
  $\widehat{u}\in \widehat{U}$, proving thus that $X$ is \SD.

  Now suppose that $X$ is \SD, i.e.\
  $\cL(X_{\widehat{u}}) \in \ID(\bR^{\widehat{u}})$ for any
  $\widehat{u}\in \widehat{U}$. Let us observe that in general
  $qA\notin \cB( \bR)^{U}$, for instance if $q=0$,
  $qA=\{ 0^{U}\} \notin \cB(\bR)^{U}$ if $U$ is uncountable. Thus, we
  firstly verify that for any $q>0$ we have
  $qA\in \cB( \bR)^{U}$. Define
  \begin{equation*}
    \cA_{q}
    \coloneqq \left\{A\in \cB(\bR)^{U}:qA\in \cB(\bR)^{U}\right\}.
  \end{equation*}
  Observe that $\cC$, the set of the cylinders in $\cB(\bR)^{U},$
  belongs to $\cA_{q}$. Moreover, $\cA_{q}$ is a
  $\sigma $-algebra. Indeed, obviously $\bR^{U}\in \cA_{q}.$ Due to
  \begin{align*}
    q(A\cap B) ={} & qA\cap qB;
    \\
    q(A\cup B) ={} & qA\cup qB,
  \end{align*}
  it follows easily that $\cA_{q}$ is closed under complements and
  countable unions. This shows that $qA\in \cB( \bR)^{U}$ for any
  $q>0$ and $A\in \cB( \bR)^{U}$. To prove \eqref{eqn2}, fix $q>1$ and
  define
  \begin{equation*}
    \cA_{q}^{\nu }
    \coloneqq \left\{ A\in \cB(\bR)^{U}:\nu(qA) \leq \nu(A) \right\} .
  \end{equation*}
  Thanks to Lemma~\ref{lemma2} (see the appendix) and the Monotone
  Class Theorem, we only need to check that
  $\cB_{0}\cup \{ \bR^{U}\} \subset \cA_{q}^{\nu }$, where $\cB_{0}$
  is as in~\ref{lemma2}. In view that $\nu$ is the master L\'{e}vy
  measure of $X$, we have that $\nu$ does not charge zero, or
  equivalently, it is $\sigma$-finite, so without loss of generality
  we can assume that $\nu$ is finite.  Clearly
  $\bR^{U}\in \cA_{q}^{\nu }$. Moreover, by consistency and equation
  \eqref{eqn1.3.8} we see that for any $A_{0}\in \cB_{0}$
  \begin{align*}
    \nu(qA_{0}) ={} & \nu [ q\pi_{\widehat{u}}^{-1}(A\setminus 0^{\widehat{u}}) ]
    \\
    \leq{}          & \nu [ \pi_{\widehat{u}}^{-1}(qA\setminus 0^{\widehat{u} }) ]
    \\
    ={}             & \nu_{\widehat{u}}(qA)
    \\
    \leq{}          & \nu_{\widehat{u}}(A)
    \\
    ={}             & \nu(\pi_{\widehat{u}}^{-1}(A\setminus 0^{\widehat{u} }))
    \\
    ={}             & \nu(A_{0}) .
  \end{align*}
  provided that
  $A_{0}=\pi_{\widehat{u}}^{-1}(A\setminus 0^{\widehat{u} })$ for some
  $\widehat{u}\in \widehat{U}$ and $A\in \cB(\bR^{\widehat{u}})$. The
  inequality above follows from the selfdecomposability of the
  finite-dimensional distributions, i.e.\ the system of
  finite-dimensional L\'{e}vy measures
  $\{ \nu_{F}\}_{\widehat{ u}\in \widehat{U}}$ fulfills \eqref{eqn2}
  on $\cB(\bR^{ \widehat{u}})$ for any
  $\widehat{u}\in \widehat{U}$. Therefore
  $\cB_{0}\cup \{ \bR^{U}\} \subset \cA_{c}^{\nu }$. This completes
  the proof.
\end{proof}

\begin{remark}
  \label{sdnonsigfin}
  From the proof of the previous proposition, if there is a pseudo
  master L\'{e}vy measure (i.e.\ it may not fulfill \eqref{eqn1.3.6})
  of $X$ satisfying \eqref{eqn2}, then the process $X$ is
  selfdecomposable.  However, the reverse may not be true in general.
\end{remark}

\section{Selfdecomposability of Volterra fields}
\label{section4}

In this section we study the selfdecomposability of \ID Volterra
fields induced by a L\'{e}vy basis. In particular, we show that under
some conditions on the kernel, the selfdecomposability of the field is
equivalent to the selfdecomposability of the L\'{e}vy basis.

Let $(L_{t})_{t\in \bR}$ be a $\bR^{n}$-valued two-sided L\'{e}vy
process and $f$ a measurable function. It is well known that the
mapping $\cL(L_{1})\mapsto \cL(\int_{\bR}f(s)dL_{s})$ is not in
general one-to-one (e.g.~\cite{BNRT08}). Note that
$\cL(\int_{\bR}f(s)dL_{s})$ corresponds to the marginal distribution
of the stationary process $X_{u}\coloneqq \int_{\bR}f(u-s)dL_{s}$
. There are several important classes of infinitely divisible
distributions that can be characterized using such mapping. Perhaps
the most important example corresponds to the class $\SD(\bR^{n})$ of
selfdecomposable distributions on $\bR^{n}$. In this case
$\cL(L_{1})\mapsto \cL(\int_{0}^{\infty }e^{-s}dL_{s})$ creates a
bijection between the class of \ID distributions on $\bR^{n}$ whose
L\'{e}vy measure has log-moment outside of zero and the class
$\SD(\bR^{n})$. Moreover, $\cL(\int_{0}^{\infty }e^{-s}dL_{s})$ is the
marginal distribution of a stationary OU process driven by
$L$. Observe that in this case
$\cL(\int_{0}^{\infty }e^{-s}dL_{s})\in \SD(\bR^{n})$ even if
$\cL(L_{1})\notin \SD(\bR^{n})$. Nevertheless, as it has been shown
in~\cite[Theorem 3.4]{BNMS06a}, this is not true for \SD fields, e.g.\
the OU process is \SD if and only if $L$ is \SD as well. We generalize
such result for \ID Volterra fields.

\subsection{The master L\'{e}vy measure of an \ID Volterra field}

In this part we investigate the master L\'{e}vy measure of certain
class of infinitely divisible fields which can be expressed in terms
of stochastic integrals, namely Volterra fields.

For the rest of the section $p\geq 0$ is such that
$\bE(|L(A)|^{p})<\infty$ for all $A\in \cR$. Recall that the
stochastic integral $\int_{\cS}f(s)L(ds)$ is well defined and has
finite $p$-moment if and only if $f\in \cL_{\Phi_{p}}.$ Here
$\cL_{\Phi_{0}}$ is the space of $L$-integrable functions. Let $L$ be
a Poissonian L\'{e}vy basis on $(\cS,\cR)$ with quadruplet
$(\gamma(s),0,\rho (s,dx),c(ds)).$ An \ID \emph{Volterra field} driven
by $L$ is a field
\begin{equation}
  X_{u}\coloneqq \int_{\cS}f(u,s)L(ds),
  \qquad u\in U,
  \label{eqn1}
\end{equation}
where $U$ is a separable space and $f\colon U\times \cS\rightarrow\bR$
a measurable function such that $f(u,\nullarg)\in \cL_{\Phi_{0}}$ for
all $u\in U$. Note that the expression in \eqref{eqn1} is also called
spectral representation of an infinitely divisible process. The next
proposition describes the master L\'{e}vy measure of $X$. Recall that
a function $g\colon U\rightarrow \bR$ is lower (upper) continuous if
$\lim \inf_{u\rightarrow u_{0}}g(u)\geq g(u_{0})$
($\lim \sup_{u\rightarrow u_{0}}g(u)\leq g(u_{0})$) for any
$u_{0}\in U$.

\begin{proposition}
  \label{propmasterlevmeas2}
  Let $X$ be as in \eqref{eqn1} with $L$ a Poissonian L\'{e}vy basis
  with quadruplet $(\gamma(s),0,\rho (s,dx),c(ds))$. Define
  \begin{equation}
    \nu \coloneqq \eta \circ g^{-1},
    \label{eqn1.3.5}
  \end{equation}
  where $g\colon\bR\times \cS\rightarrow\bR^{U}$ is the function
  defined as $g_{u}(x,s)=xf(u,s),$ $u\in U$ and
  $\eta (dxds)=\rho(s,dx)c(ds)$. Suppose that $f(\nullarg ,s)$ is
  non-identically zero and lower or upper continuous for $c$-almost
  all $s\in \cS$. Then $\nu$ as in \eqref{eqn1.3.5} is the master
  L\'{e}vy measure of $X$.
\end{proposition}

For the proof of this proposition we need the next result:

\begin{proposition}
  \label{propmasterlevmeas}
  Let $X$ be as in \eqref{eqn1} with $L$ a Poissonian L\'{e}vy
  basis. Then, for any $\widehat{u}\in \widehat{U},$ $X_{\widehat{u}}$
  has characteristic triplet
  $(\Gamma_{\widehat{u}},0,\nu_{\widehat{u}})$ with
  \begin{equation*}
    \Gamma_{\widehat{u}}
    = \int_{\cS}\bigl\{ \gamma(s)\pi_{\widehat{u}}(f(\nullarg ,s))
    + \int_{\bR}[ \tau_{\#\widehat{u}}[\pi_{\widehat{u}}(xf(\nullarg ,s))]
    - \pi_{\widehat{u}}(f(\nullarg ,s))\tau_{1}(x)] \rho(s,dx)\bigr\} c(ds),
  \end{equation*}
  and
  \begin{equation}
    \nu_{\widehat{u}}
    =\nu \circ \pi_{\widehat{u}}^{-1},
    \qquad\text{on }\cB(\bR^{\widehat{u}}\setminus 0^{\widehat{u}}),
    \label{eqn1.3.5.1}
  \end{equation}
  with $\nu$ given by \eqref{eqn1.3.5}.
\end{proposition}

\begin{proof}
  Let us start by noting that since the mapping
  $(x,s)\mapsto g_{u}(x,s)$ is
  $\cB(\bR)\otimes \cB_{\cS}/\cB(\bR)$-measurable for all $u\in U$ we
  have that $\nu$ is well defined. Now, let
  $\widehat{u}\in \widehat{U}$ and observe that for all
  $\theta \in \bR^{\widehat{u}}$
  \begin{equation*}
    \langle X_{\widehat{u}},\theta \rangle
    = \int_{\cS}\langle \pi_{\widehat{u}}[ f(\nullarg ,s)] ,\theta \rangle L(ds).
  \end{equation*}
  Thus, from \eqref{eqn0.1}, the cumulant function of
  $X_{\widehat{u}}$ satisfies
  \begin{equation*}
    \rC\{\theta \ddagger X_{\widehat{u}}\}
    =\int_{\cS}\psi \left(
      \langle \pi_{\widehat{u}}[ f(\nullarg ,s)] ,\theta \rangle,s
    \right) c(ds),
  \end{equation*}
  with $\psi$ as in \eqref{eqn1.1}. But for $c$-almost all $s\in \cS$
  \begin{align*}
    \MoveEqLeft[3]\psi(\langle \pi_{\widehat{u}}[ f(\nullarg ,s)]
    ,\theta \rangle,s)
    \\
    ={} & i\langle \gamma(s)\pi_{\widehat{u}}[f(\nullarg ,s)],\theta
          \rangle +\int_{\bR} \left[
          e^{ix\langle\pi_{\widehat{u}}[f(\nullarg ,s)] ,\theta \rangle }
          -1-i\langle \pi_{\widehat{u}}[ f(\nullarg ,s)],\theta \rangle
          \tau_{1}(x)\right] \rho(s,dx)
    \\
    ={} & i\Bigl\langle \gamma(s)\pi_{\widehat{u}}[ f(\nullarg ,s)]
          +\int_{\bR}\{
          \tau_{\#\widehat{u}}[\pi_{\widehat{u}}(xf(\nullarg,s))]
          -\pi_{\widehat{u}}(f(\nullarg ,s))\tau_{1}(x)\} \rho
          (s,dx),\theta\Bigr\rangle
    \\
        & +\int_{\bR}\left\{ e^{i\langle \pi_{\widehat{u}}[
          g_{\nullarg}(x,s)] ,\theta \rangle} -1-i\langle
          \tau_{\#\widehat{u}}[\pi_{\widehat{u}}(g_{\nullarg
          }(x,s))],\theta \rangle \right\} \rho(s,dx).
\end{align*}
Integrating the previous equation with respect to $c$ and invoking the
uniqueness of the triplet, the result follows.
\end{proof}

\begin{proof}[Proof of Proposition~\ref{propmasterlevmeas2}]
  From Theorem~\ref{thmmasterlevymeasure} and the previous proposition
  we only need to check that $\nu$ does not charge zero. Let $U_{0}$
  be a dense set in $U$, then from the definition of $\nu$
  \begin{align*}
    \nu(\pi_{U_{0}}^{-1}(0^{U_{0}}))
    ={} & \eta(\{(s,x) :xf(u,s) =0\ \forall \ u\in U_{0}\})                                                 \\
    ={} & \eta(\{(s,x) :x=0\text{ or }f(u_{0},s) =0\ \forall \ u\in U_{0}\})                                \\
    ={} & \lim_{n\rightarrow \infty }\eta(\{(s,x) :f(u_{0},s) =0\ \forall \ u\in U_{0},\abs{x} >1/n\}) \\
    ={} & \lim_{n\rightarrow \infty }\int_{\{ s\,:\,f(u,s) =0\: \forall \:
          u\in U_{0}\} }\rho(s,\{ \abs{x} >1/n\}) c(ds) =0,
  \end{align*}
  because if $f(u,s) =0$ $\forall$ $u\in U_{0},$ by the
  lower$\setminus$upper continuity we have that $f(u,s) =0$ for all
  $u\in U$, which is contradictory.
\end{proof}

\begin{remark}
  \label{remarknonuniquemasterlevymeasure}
  Observe that equation \eqref{eqn1.3.5.1} holds for any \ID Volterra
  field. This means that the measure $\nu$ defined by \eqref{eqn1.3.5}
  is always a pseudo master L\'{e}vy measure of an \ID Volterra
  field. However, it is not clear that in general such a measure does
  not charge zero for general index set $U$.
\end{remark}

\begin{remark}
  Note that $\nu$ can be viewed as a general
  $\Upsilon^{0}$-transformation (see for
  instance~\cite{BNPAT13}). Indeed, let $(E,\cB_{E})$ be a measurable
  space. For any fixed measurable function
  $g\colon\cS\times \bR^{d}\rightarrow E$ and $c$ a $\sigma$-finite
  measure on $\cS$, define the functional
  \begin{equation*}
    \Upsilon_{g,c}^{0}(\rho)(A)
    \coloneqq \int_{\cS}\int_{\bR^{d}}\indf_{g^{-1}(A)}(x,s)\rho(s,dx)c(ds),
    \qquad A\in \cB_{E},
  \end{equation*}
  where $(\rho(s,dx))_{s\in \cS}$ is a measurable collection of
  measures. Notice that $\Upsilon_{g,c}^{0}(\rho)=\eta \circ g^{-1}$,
  with $\eta(dxds)=\rho(s,dx)c(ds)$. In particular, if $\cS=\bR^{+}$,
  $g(x,s)=xs$ and $\rho$ does not depend on $s,$ $\Upsilon_{g,c}^{0}$
  coincides with the usual $\Upsilon^{0}$-transformation of $\rho$ via
  $c$.  In this case, it is well known that such transformation is
  generally not one-to-one. More generally, if $g(x,s)=T(s)x$, where
  $T$ is a measurable collection of linear mappings on $\bR^{d}$, we
  have that $\Upsilon_{g,c}^{0}(\rho)$ is the L\'{e}vy measure of
  $\Upsilon_{T}(\mu)$, the probability measure with cumulant
  \begin{equation*}
    \rC\{\theta \ddagger \Upsilon_{T}(\mu)\}
    = \int_{\cS}\rC\{T(s)\theta \ddagger L^{\prime }\}c(ds),
    \qquad L^{\prime }\sim \mu \text{ and }\mu \in \ID(\bR^{d}).
  \end{equation*}
  with $L^{\prime }$ the L\'{e}vy seed of a factorizable L\'{e}vy
  bases.  Hence, $\Upsilon_{g,c}^{0}(\rho)$ can be viewed as the
  L\'{e}vy measure of the probability measure
  $\Upsilon_{f}(\mu)\in \ID(\bR^{U}),$ with $\mu$ being the
  distribution of $L$ in $\ID(\bR^{\cR})$, characterized by
  \begin{equation*}
    \rC\{\theta\ddagger\Upsilon_{f}(\mu)\circ\pi_{\widehat{u}}^{-1}\}
    = \int_{\cS}\rC\{\langle \pi_{\widehat{u}}\left[
      f(\nullarg ,s)\right] ,\theta \rangle \ddagger \mu(s)\}c(ds)
    \qquad\text{for any $\widehat{u}\in \widehat{U},
      \theta \in \bR^{\widehat{u}}$},
  \end{equation*}
  and with $\mu(s)$ the distribution of the L\'{e}vy seed
  $L^{\prime }(s)$ for all $s\in \cS$. See~\cite{BNRT08}
  and~\cite{BNPAT13} for an extensive discussion on
  $\Upsilon^{0}$-transformations and generalizations.
\end{remark}

\subsection{Inherited selfdecomposability from the L\'{e}vy basis}

Typically the selfdecomposability of stochastic integrals can be
obtained by assuming that the integrator is also selfdecomposable. In
the case of random fields such result is also true. In this part we
verify this property for $\ID$ Volterra fields by using the
characterization provided in Proposition~\ref{propdilatationsd}. The
converse of such property will be discussed in the next subsection.

\begin{proposition}
  \label{ifpartthmselfd}
  Let $X$ be as in \eqref{eqn1} with $L=\{ L(A) :A\in \cR\}$ a
  L\'{e}vy basis. Suppose that $\cL( L) \in \bL_{m}(\bR^{\cR }) ,$
  then the law of $X$ belongs to $\bL_{m}(\bR^{U})$.
\end{proposition}

\begin{proof}
  We will only check the case $m=0$, the general case follows by
  induction.  Fix $q>1$ and suppose that
  $\cL(L) \in \bL_{0}(\bR^{\cR}) .$ Firstly, let us observe that in
  this case for any non-negative measurable function
  $h\colon\cS\times \bR\rightarrow \bR$, we have
  \begin{equation}
    \int_{\cS}\int_{\bR}h(q^{-1}x,s) \rho(s, dx) c(ds) \leq
    \int_{\cS}\int_{ \bR}h(x,s) \rho(s, dx) c(ds) .
    \label{eqn1.3.10}
  \end{equation}
  Indeed, since $\cL(L) \in \bL_{0}( \bR^{\cR})$, it follows that
  $\cL(L(A)) \in \bL_{0}(\bR)$ for all $A\in \cR$. Therefore,
  $\nu_{A}(qB) \leq \nu_{A}(B)$ for any $A\in \cR$ and
  $B\in \cB(\bR)$, where $\nu_{A}(\nullarg)$ is the L\'{e}vy measure
  of $L(A) .$ But in view of
  $\nu_{A}(B) =\int_{\cS}\int_{ \bR^{d}}\indf_{A\times B}(x,s) \rho(s, dx) c(ds) $,
  \eqref{eqn1.3.10} holds for every function of the form
  $\indf_{A\times B}$ with $A\in \cR$ and $B\in \cB(\bR)$. The general
  case follows by the Functional Monotone Class Theorem.

  Thanks to Proposition~\ref{propdilatationsd} and
  Remark~\ref{sdnonsigfin}, in order to show that
  $\cL(X) \in \bL_{m}(\bR^{U})$ it is enough to check that
  \eqref{eqn2} holds for some pseudo master L\'{e}vy measure of $X$
  (its existence is guaranteed by
  Theorem~\ref{thmmasterlevymeasure}). Let $\nu$ be as in
  Proposition~\ref{propmasterlevmeas2}, then $\nu$ is a pseudo master
  L\'{e}vy measure of $X$ and
  \begin{align*}
    \nu(qA)
    ={}    & \int_{\cS}\int_{\bR}\indf_{qA} [ g(x,s) ]\rho(s, dx) c(ds)
    \\
    ={}    & \int_{\cS}\int_{\bR}\indf_{A}[ q^{-1}xf(\nullarg ,s) ]\rho(s, dx) c(d s)
    \\
    ={}    & \int_{\cS}\int_{\bR}\indf_{A}[ g(q^{-1}x,s) ] \rho(s, dx) c(d s)
    \\
    \leq{} & \int_{\cS}\int_{\bR}\indf_{A}[ g(x,s) ] \rho(s, dx) c(d s)
    \\
    ={} & \nu(A) ,\qquad A\in \cB(\bR)^{U},
  \end{align*}
  where we used \eqref{eqn1.3.10}. Thus,
  $\cL(X) \in \bL_{0}(\bR^{U})$.
\end{proof}

\subsection{Identification problem for \ID Volterra fields and
  selfdecomposability}

As we showed in the previous subsection, in general, if the L\'{e}vy
basis is \SD, the associated Volterra process is also \SD. In this
part we give sufficient conditions for which the converse holds. Based
on~\cite{Sau14}, we show that such conditions can be checked easily
for the class of stationary Volterra \ID fields.

Let $X$ be as in \eqref{eqn1} with $L$ a Poissonian L\'{e}vy basis
with characteristic quadruplet $(\gamma(s) ,0,\rho(s, d x) ,c( ds)) .$
Recall that $p\geq 0$ is such that $\bE(\abs{L(A)}^{p}) <\infty$ for
all $A\in \cR$. Define
$S_{\Phi_{p}}( f) \coloneqq \olspan\{ f(u,\nullarg) \}_{u\in U}$ in
$\cL_{\Phi_{p}}$ and
$S_{_{p}}(X) \coloneqq \olspan\{ X_{u}\}_{u\in U}$ in
$\cL^{p}(\Omega)$. In order to present the main theorem of this
section we need to introduce the following condition:

\begin{condition}
  \label{conditionkernel}
  For any $A\in \cR$, we have $\indf_{A}\in S_{\Phi_{p}}(f)$ (or
  equivalently $S_{\Phi_{p}}(f)=\nobreak \cL_{\Phi_{p}}$).
\end{condition}

\begin{remark}
  \label{remarkl2total}
  Note that when $p=2$ (i.e.\ $L$ is square-integrable) and $L$ is
  centered and homogeneous, Condition~\ref{conditionkernel} is
  equivalent to $S_{\Phi_{2}}(f) =\cL^{2}(\bR,ds) .$
\end{remark}

\begin{theorem}
  \label{propsdifpart}
  Let $X$ be as in \eqref{eqn1} with $L=\{ L(A) :A\in \cR\}$ a
  L\'{e}vy basis with characteristic quadruplet
  $(\gamma(s) ,b(s) ,\rho(s, dx) ,c(ds))$. Suppose that
  $\cL(L) \in \bL_{m}(\bR^{\cR }) ,$ then the law of $X$ belongs to
  $\bL_{m}(\bR^{U}) $. Conversely, assume that
  Condition~\ref{conditionkernel} holds. Then
  $\cL(X) \in \bL_{m}(\bR^{U})$ implies that
  $\cL(L) \in \bL_{m}( \bR^{\cR})$.
\end{theorem}

For the proof of this theorem we need some auxiliary results:

\begin{lemma}
  \label{lemmalineartransf}
  Suppose that $X\sim \mu$ with $\mu \in \bL_{m}(\bR^{d})$. Then for
  any linear transformation $T\colon \bR^{d}\rightarrow \bR^{k}$, the
  law of $T(X)$ is in $\bL_{m}(\bR^{k})$.
\end{lemma}

\begin{proof}
  The proof is straightforward thus omitted.
\end{proof}

\begin{proposition}
  \label{propspanfspanX}
  Assume that $\indf_{A}\in S_{\Phi_{p}}( f)$ for some
  $A\in \cR$. Then $L(A) \in S_{_{p}}(X) $. Conversely, if $L$ is
  symmetric (or centered) and $L(A) \in S_{_{p}}(X)$ for some
  $A\in \cR$, then $\indf_{A}\in S_{\Phi_{p}}(f)$.
\end{proposition}

\begin{proof}
  The result follows from the continuity of the mapping
  \begin{align*}
    (f\in \cL_{\Phi_{p}}) \longmapsto \Bigl(\int_{\cS}f(
    s) L(ds) \in \cL^{p}(\Omega ,\cF,\bP) \Bigr)
  \end{align*}
  and the fact that when $L$ is symmetric or centered such mapping is
  in fact an isomorphism.
\end{proof}

\begin{proof}[Proof Theorem~\ref{propsdifpart}]
  We will only check the case $m=0$, the general case follows by
  induction.  The first part was already proved in
  Proposition~\ref{ifpartthmselfd}.  Suppose that
  $\cL(X)\in \bL_{0}(\bR^{U})$ and Condition~\ref{conditionkernel}
  holds. Let $A_{1},\ldots ,A_{k}\in \cR$. From
  Condition~\ref{conditionkernel} and
  Proposition~\ref{propspanfspanX}, for any $j=1,\ldots k$ there are,
  $\mathbf{\theta}_{j}^{n}\coloneqq(\theta_{i}^{j})_{i=1}^{n}\in \bR^{n}$
  and
  $\widehat{u}_{j}^{n}\coloneqq (u_{i}^{j})_{i=1}^{n}\subset U^{n}$
  with $n\in \bN$, such that
  $\langle \smash{\mathbf{\theta}_{j}^{n},X_{\widehat{u}_{j}^{n}}}\rangle \xrightarrow{\smash{\scriptstyle \cL^{p}(\Omega)}} L(A_{j})$
  for any $j=1,\ldots k$. Putting
  $u^{n}\coloneqq \bigcup_{j=1}^{k}\widehat{u}_{j}^{n}$ we have that
  there exists $M(\theta),$ a $k\times \#u^{n}$ matrix only depending
  on $\mathbf{\theta }_{j}^{n}$ for $j=1,\ldots k,$ such that
  \begin{equation}
    M(\theta)X_{u^{n}}
    =(\langle \mathbf{\theta}_{j}^{n},X_{\widehat{u}_{j}^{n}}\rangle)_{j=1}^{k}
    \xrightarrow{\cL^{p}(\Omega)} (L(A_{j}))_{j=1}^{k}
    \qquad\text{as }n\rightarrow \infty .
    \label{eqn1.3.11}
  \end{equation}
  From Lemma \eqref{lemmalineartransf} and the fact that
  $\cL(X)\in \bL_{0}(\bR^{U})$, we have
  $\cL[M(\theta)X_{u^{n}}]\in \bL_{0}(\bR^{k})$ for any $n\in \bN$.
  The closedness of $\bL_{0}(\bR^{k})$ under weak limits guarantees
  that the weak limit of $\cL[M(\theta)X_{u^{n}}]$ belongs to
  $\bL_{0}(\bR^{k})$, or in other words
  $\cL((L(A_{j}))_{j=1}^{k})\in \bL_{0}(\bR^{k}),$ the
  selfdecomposability of $L$.
\end{proof}

\begin{remark}
  In view of Proposition~\ref{propdilatationsd} and equation
  \eqref{eqn1.3.11}, Condition~\ref{conditionkernel} allows us to
  determine the L\'{e}vy basis through the process $X$ by linear
  approximations, i.e.\ under this assumption for some $p\geq 0$
  \begin{equation}
    \olspan\{ X_{u}\}_{u\in U}=\olspan\{ L(A) \}_{A\in\cR}
    \qquad\text{in }\cL^{p}(\Omega) .
    \label{eqn1.3.14}
  \end{equation}
  Thus, this can be considered as an identification condition also
  discussed in more depth in~\cite{Sau14}.
\end{remark}

Due to~\cite{Sau14}, in the stationary case,
Condition~\ref{conditionkernel} can be easily checked as the following
theorem shows:

\begin{theorem}
  \label{sdstationary}
  Let $L$ be an homogeneous L\'{e}vy basis on $\cB_{b}(\bR^{d})$ and
  $g\in \cL^{1}( \bR^{d},ds) \cap \cL_{\Phi_{0}}$ having non-vanishing
  Fourier transform. Then the law of the \ID Volterra field
  \begin{equation*}
    X_{u}\coloneqq \int_{\bR^{d}}g(u-s) L(ds) ,
    \qquad u\in \bR^{d},
  \end{equation*}
  belongs to $\bL_{m}(\bR^{\bR^{d}})$ if and only if
  $\cL(L) \in \bL_{m}(\bR^{ \cB_{b}(\bR^{d}) })$.
\end{theorem}

\begin{proof}
  From Theorem 13 in~\cite{Sau14}, we have that in this case
  $S_{\Phi_{0}}(g)=\cL_{\Phi_{0}}$, which implies that
  Condition~\ref{conditionkernel} is fulfilled. The result follows
  from this and the previous theorem.
\end{proof}

\begin{example}[Ornstein-Uhlenbeck processes]
  \label{exampleOU}
  Suppose that $L$ is a L\'{e}vy process with characteristic triplet
  $(\gamma ,b,\rho)$. Let
  \begin{equation*}
    f(u,s)
    =\varphi_{0}(u-s) \coloneqq e^{-(u-s) } \indf_{\{ s\leq u\} },
    \qquad u,s\in \bR.
  \end{equation*}
  The resulting \ID Volterra field is the classic OU process driven by
  $L$.  It is well known that such processes are well defined if and
  only if $\int_{\abs{x} >1}\log(\abs{x}) \rho(dx) <\infty .$
  Moreover, in this case $\cL(X_{u}) \in \bL_{0}(\bR)$ for all
  $u\in \bR$ and it is uniquely determined by $L$ and vice versa.
  But, since $\widehat{\varphi_{0}}$, the Fourier transform of
  $\varphi_{0}$, never vanishes, we conclude that an OU process is \SD
  if and only if the background L\'{e}vy process is \SD as well, just
  as in~\cite[Theorem 3.4]{BNMS06a}.
\end{example}

\begin{example}[\lss process with a Gamma kernel]
  \label{examplegamma}
  Take $L$ to be a L\'{e}vy process with characteristic triplet
  $(\gamma ,b,\rho)$. Let $\alpha >-1$ and consider
  \begin{equation}
    f(u,s)
    =\varphi_{\alpha }(u-s)
    \coloneqq e^{-(u-s) }(u-s)^{\alpha }\indf_{\{ s\leq u\} } ,
    \qquad u,s\in \bR.
    \label{eqn1.3.15}
  \end{equation}
  It has been shown in~\cite{BOC14}, that
  $f(u,\nullarg)\in \cL_{\Phi_{0}}$ for every (equivalently for some)
  $u\in \bR$ if and only if the following two conditions are
  satisfied:

  \begin{enumerate}
  \item $\int_{\abs{x} >1}\log(\abs{x}) \rho(dx)
    <\infty$,
  \item One of the following conditions holds:
    \begin{enumerate}
    \item $\alpha >-1/2;$
    \item $\alpha =-1/2$, $b=0$ and
      $\int_{\abs{x} \leq 1}\abs{x}^{2}\abs{\log(\abs{x})} \rho(dx) <\infty ;$
    \item $\alpha \in(-1,-1/2)$, $b=0$ and
      $\int_{\abs{x} \leq 1}\abs{x}^{-1/\alpha }\rho(dx) <\infty .$
    \end{enumerate}
  \end{enumerate}

  Moreover, according to~\cite{PedSau14}, for any $-1<\alpha <0$,
  $\cL(X_{u}) \in \bL_{0}(\bR)$ for all $u\in \bR$.  However, since
  the Fourier transform of $\varphi_{\alpha }$ is given by
  \begin{equation*}
    \widehat{\varphi_{\alpha }}(\xi)
    =\frac{\Gamma(\alpha +1) }{\sqrt{2\pi }}(1-i\xi)^{-\alpha -1},
    \qquad \xi \in \bR,
  \end{equation*}
  the law of the L\'{e}vy semistationary process
  $X_{u}=\int_{-\infty }^{u}e^{-(u-s) }(u-s)^{\alpha }dL_{s}$ is in
  $\bL_{0}( \bR^{U})$ if and only if $L$ is selfdecomposable.
\end{example}

\begin{example} [Fractional L\'{e}vy motions]
  \label{examplefractional}
  Suppose that $L$ is a centered and square-integrable L\'{e}vy
  process with characteristic triplet $(\gamma ,b,\rho)$. For
  $\alpha \in(0,1/2)$ consider
  \begin{equation*}
    f(u,s) =(u-s)_{+}^{\alpha }-(-s)_{+}^{\alpha },\qquad u,s\in
    \bR,
  \end{equation*}
  where $(x)_{+}$ denotes the positive part of $x$. In~\cite{CM11}, it
  was shown that $\{ f(u,\nullarg) \}_{u\in U}$ is total in
  $\cL^{2}(ds)$, which according to Remark~\ref{remarkl2total},
  implies Condition~\ref{conditionkernel}.  Furthermore, the authors
  also noted that in general the marginal distribution of the Volterra
  process induced by this function is not selfdecomposable unless $L$
  is \SD, as Theorem~\ref{propsdifpart} shows.
\end{example}

\section{Integrated \ID Volterra fields}
\label{sec:integr-id-volt}

In this section we are interested in the random variable
\begin{equation}
  \mu(X;A) =\int_{A}X_{u}\mu(du) ,
  \qquad A\in \cB_{b}(\mu) ,
  \label{eqn5.1}
\end{equation}
where $X$ is an \ID Volterra field, $\mu$ a $\sigma$-finite measure
and $\cB_{b}(\mu) \coloneqq \{ A:\mu(A) <\infty \}$. We will consider
the following associated field
\begin{equation}
  X^{\mu }=(\mu(X;A))_{A\in \cB_{b}(\mu) },
  \label{eqn5.2}
\end{equation}
We start by giving sufficient conditions for $\mu(X;A)$ to exists.

\subsection{Existence of $\mu(X;A)$}

In this part we present sufficient conditions for which $\mu(X;A)$ as
in \eqref{eqn5.1} exists. To do this we use the Stochastic Fubini
Theorem presented in~\cite{BNBOC11}.

Let $(U,\cB(U),\mu)$ be a measurable space, where $U$ is a Polish
space, i.e.\ a complete and separable metric space, $\cB(U)$ its Borel
$\sigma$-algebra and $\mu$ a $\sigma $-finite measure. Note that
defining $\mu(X;A)$ involves two issues. Firstly, we must to verify
that the process $X$ has at least a measurable modification with
respect to $\cF\otimes\cB(U)$. The second consists in providing
sufficient conditions which guarantee that
$X\in \cL^{1}(U,\cB(U),\mu)$. In particular, for \ID Volterra fields,
it would be desirable to relate such conditions directly to the
kernel. Let $X$ be as in \eqref{eqn1}. In this case, it has been shown
in \cite{BNBOC11} that $X$ always admits a measurable
modification. Furthermore, the following Stochastic Fubini Theorem for
L\'{e}vy bases provides sufficient conditions for
$X\in \cL^{1}(U,\cB(U),\mu)$.

\begin{theorem}[{Stochastic Fubini
    Theorem \cite[Theorem~3.1]{BNBOC11}}]
  \label{fubini}
  Let $L$ be a centered L\'{e}vy basis with characteristic triplet
  $(\gamma(s),b(s),\rho(s,dx),c(ds)).$ Consider
  $f\colon U\times \cS\rightarrow\bR$ be a
  $\cB(U)/\cB_{\cS}$-measurable function such that
  $f(u,\nullarg)\in \cL_{\Phi_{0}}$ for all $u\in U$. Assume that for
  $A\in \cB(U)$
  \begin{equation}
    \int_{A}\norm{f(u,\nullarg)}_{\Phi_{1}}\mu(du)<\infty .
    \label{eqn5.4}
  \end{equation}
  where $\norm{\nullarg}_{\Phi_{1}}$ is as in \eqref{eqn0.3}. Then
  $f(\nullarg ,s)\in \cL^{1}(U,\cB(U),\mu)$ for $c$-almost every
  $s\in \cS$ and the mapping $s\longmapsto$ $\int_{A}f(u,s)\mu(du)$
  belongs to $\cL_{\Phi_{1}}$. In this case, all the integrals below
  exist and almost surely
  \begin{equation*}
    \int_{A}\left[ \int_{\cS}f(u,s)L(ds)\right] \mu(du)
    = \int_{\cS}\left[ \int_{A}f(u,s)\mu(du)\right] L(ds).
  \end{equation*}
  Moreover, if $\mu$ is finite, \eqref{eqn5.4} is equivalent to
  \begin{equation}
    \int_{A}\int_{\cS}\Bigl[ f^{2}(u,s)b^{2}(s)+\int_{\bR}|xf(u,s)|\wedge |xf(u,s)|^{2}\Bigr] c(ds)\mu(du)<\infty .  \label{eqn5.5}
  \end{equation}
\end{theorem}

In spirit of the previous theorem, for the rest of this section $L$
will be assumed to be centered.

\begin{remark}
  \label{remarkstationaryfubini}
  Note that in the stationary case, i.e.\ when $L$ is homogeneous and
  $f(u,s) =g(u-s) ,$ with $g\in \cL_{\Phi_{0}}$, \eqref{eqn5.4} holds
  if and only if $\mu( A) <\infty$ and $g\in \cL_{\Phi_{1}}$. Indeed,
  this follows from the fact that in this case
  \begin{equation}
    \norm{f(u,\nullarg)}_{\Phi_{1}}=\norm{g}_{\Phi_{1}}
    \qquad\text{for all }u\in U.
    \label{eqn5.5.1}
  \end{equation}
\end{remark}

Using the previous theorem, it is easy to check the validity of the
next proposition:

\begin{proposition}
  \label{popintid}
  Assume that \eqref{eqn5.4} holds. Then the random variable
  $\mu(X;A)$ in \eqref{eqn5.1} is well defined and it is infinitely
  divisible.
\end{proposition}

\subsection{Selfdecomposability of $X^{\mu }$}

In this part we study the selfdecomposability of the fields $X^{\mu }$
defined in \eqref{eqn5.2}.

Let $(U,\cB(U) ,\mu)$ be a measurable space as in the previous
subsection. From Theorem~\ref{fubini}, if \eqref{eqn5.4} holds for all
$A\in \cB_{b}(\mu) ,$ the random field
$X^{\mu }=(\mu(X;A))_{A\in \cB_{b}(\mu) }$ is well defined and it
admits the following representation
\begin{equation*}
  \mu(X;A) =\int_{\cS}\mu_{f}(A,s) L(ds) ,
  \qquad A\in\cB_{b}(\mu) ,
\end{equation*}
where
\begin{equation*}
  \mu_{f}(A,s) \coloneqq \int_{A}f(u,s) \mu(d u) ,
  \qquad A\in \cB_{b}(\mu) ,s\in \cS.
\end{equation*}
In addition, $X^{\mu }$ is an \ID field with system of characteristic
triplets given as in Proposition \eqref{propmasterlevmeas}. However,
since the indexing set of the field is not separable in general, we
can only argue that the measure given in \eqref{eqn1.3.5} is a pseudo
master L\'{e}vy measure of $X^{\mu }$.

Due to Theorem~\ref{propsdifpart}, if
$\olspan(\{ \mu_{f}( A,\nullarg) \}_{A\in \cB_{b}(\mu) }) =\cL_{\Phi_{1}}$,
we have that the law of $Y_{\mu }$ is in
$\bL_{m}( \bR^{\cB_{b}(\mu) })$ if and only if
$\cL( L) \in \bL_{m}(\bR^{\cR})$. Here a natural question appears, is
the selfdecomposability of $X$ (or
$\olspan(\{ f(u,\nullarg) \}_{u\in U}) =\cL_{\Phi_{1}}$) necessary and
sufficient for the one on $X^{\mu }?$. In the stationary case the
answer is affirmative as the following theorem shows:

\begin{theorem}
  \label{theoremintegratedprocsd}
  Let $L$ be an homogeneous centered L\'{e}vy basis on
  $\cB_{b}(\bR^{d})$ and
  $g\in \cL^{1}(\bR^{d},ds) \cap \cL_{\Phi_{1}}$ having non-vanishing
  Fourier transform. Assume that $\mu$ is a finite measure such that
  $\mu \sim Leb^{d}$. Then, the law of the integrated process
  \begin{equation*}
    \mu(X;A) =\int_{A}X_{u}\mu(du) ,
    \qquad A\in \cB_{b}(U) ,
  \end{equation*}
  where
  \begin{equation*}
    X_{u}\coloneqq \int_{\bR^{d}}g(u-s) L(ds) ,
    \qquad u\in \bR^{d},
  \end{equation*}
  belongs to $\bL_{m}(\bR^{\cB_{b}(\mu) })$ if and only if
  $\cL(X) \in \bL_{m}( \bR^{\bR^{d}})$ or equivalently
  $\cL(L) \in \bL_{m}(\bR^{\cR})$.
\end{theorem}

\begin{proof}
  From the discussion above, we only need to check that
  \begin{align*}
    \olspan(\{\mu_{f}(A,\nullarg)\}_{A\in \cB_{b}(\mu})=\cL_{\Phi_{1}}.
  \end{align*}
  Suppose the opposite, this is (thanks to the Hahn-Banach Theorem),
  there exists $h$ a non-zero measurable function in the dual of
  $\cL_{\Phi_{1}}$ such that (see~\cite{RaoRen94} or Corollary~3
  in~\cite{Sau14})
  \begin{equation}
    \int_{\bR^{d}}\mu_{f}(A,s)h(s)ds=0,
    \qquad\text{for all }A\in \cB_{b}(\bR^{d}).
    \label{eqn5.6}
  \end{equation}
  Since $\mu$ is finite, $\cB_{b}(\mu )=\cB(\bR^{d})$.  Moreover, due
  to \eqref{eqn5.4} and $g\in \cL^{1}(\bR^{d},ds)\cap \cL_{\Phi_{1}},$
  we have that
  \begin{align*}
    0={} & \int_{\bR^{d}}\mu_{f}(A,s)h(s)\,ds
    \\
    ={}  & \int_{\bR^{d}}\int_{A}g(u-s)h(s)\,\mu(du)ds
    \\
    ={}  & \int_{A}\int_{\bR^{d}}g(u-s)h(s)\,ds\mu(du)
           \qquad\text{for all }A\in \cB(\bR^{d}).
  \end{align*}
  Therefore
  \begin{equation*}
    \int_{\bR^{d}}g(u-s)h(s)ds=0
    \qquad\text{for }\mu \text{-almost all }u\in U.
  \end{equation*}
  But $\mu \sim Leb^{d}$, consequently the previous equation holds for
  almost all $u\in U$. Therefore, from the proof of Theorem~13
  in~\cite{Sau14}, we obtain that $h=0,$ a contradiction.
\end{proof}

\begin{remark}
  Note that in the non-stationary case, we are able to show that
  equation \eqref{eqn5.6} implies that
  $\mu(\{ u\in U:\int_{\bR^{d}}f(u,s) h(s) c(ds) =0\}) =0$. However,
  in general it is not possible to verify from this that
  \begin{equation*}
    \int_{\bR^{d}}f(u,s) h(s) c(d s) =0
    \qquad\text{for all }u\in U,
  \end{equation*}
  which under Condition~\ref{conditionkernel} occurs if and only if
  $h=0$.
\end{remark}

\begin{example}[Ornstein-Uhlenbeck processes]
  Let $L$ be a centered L\'{e}vy process with characteristic triplet
  $(\gamma ,b,\rho)$, $f$ as in Example~\ref{exampleOU} and $\mu$
  being a $\sigma$-finite measure on $\bR$. Then \eqref{eqn5.4} holds
  for every $A\in \cB_{b}(\mu)$. Indeed, from
  Remark~\ref{remarkstationaryfubini} we only need to verify that
  $f\in \cL_{\Phi_{1}}$. This occurs (see Section 1) if and only if
  \begin{equation*}
    \int_{\bR}\int_{0}^{\infty }\left[ |xe^{-s}|\wedge |xe^{-s}|^{2}\right] ds\rho(dx)<\infty .
  \end{equation*}
  We have that
  \begin{align*}
    \MoveEqLeft[3]
    \int_{\bR}\int_{0}^{\infty }\left[
    |xe^{-s}|\wedge |xe^{-s}|^{2}\right] ds\rho(dx)
   \\
    ={} &
          \frac{1}{2}\int_{|x|\leq 1}|x|^{2}\rho(dx)
          +\int_{|x|>1}\int_{0}^{\log(|x|)}|xe^{-s}|ds\rho(dx)
    \\
        & +\int_{|x|>1}\int_{\log(|x|)}^{\infty }|xe^{-s}|^{2}ds\rho(dx)
    \\
    ={} & \frac{1}{2}\int_{\bR}1\wedge |x|^{2}\rho(dx)
          +\int_{|x|>1}[\frac{|x|^{2}-1}{|x|}]\rho(dx)<\infty ,
  \end{align*}
  due to the fact that $\int_{|x|>1}|x|\rho(dx)<\infty$ (because $L$
  has first moment). Therefore, the integrated process
  \begin{equation*}
    X^{\mu }(A)=\int_{A}X_{u}\mu(du),
    \qquad A\in \cB_{b}(\mu),
  \end{equation*}
  is well defined. In particular, if $A=[0,t]$ and $\mu(du)=du$, we
  get that
  \begin{align*}
    X_{t}^{\mu }
    \coloneqq{} &  X^{\mu }([0,t])
    \\
    = {} &\int_{0}^{t}X_{u}\,du
    \\
    = {} &L_{t}-(X_{t}-X_{0}),
           \qquad t\geq 0,
  \end{align*}
  the Langevin equation. Since $(X_{t}-X_{0})$ is independent of
  $L_{t}$, $(X_{t}^{\mu })_{t\geq 0}$ is \SD if and only if $X$ is \SD
  or $L$ is \SD.  Note that this result is true in general for any
  L\'{e}vy process for which $\int_{|x|>1}\log(|x|)\rho(dx)<\infty$,
  which means that the condition on $L$ in
  Theorem~\ref{theoremintegratedprocsd} is sufficient but not
  necessary.
\end{example}

\begin{example} [\lss process with a Gamma kernel]
  Let $L$ and $\mu$ be as in the previous example. Consider $f$ as in
  Example~\ref{examplegamma}. We want to check that \eqref{eqn5.4}
  holds for any $\mu$-bounded set. To do this, we observe that
  $\varphi_{\alpha }\in \cL_{\Phi_{0}}$ if and only if
  $\varphi_{\alpha }\in \cL_{\Phi_{1}}$, where $\varphi_{\alpha }$ is
  as in \eqref{eqn1.3.15}. Obviously if
  $\varphi_{\alpha }\in \cL_{\Phi_{1}}$ we have that
  $\varphi_{\alpha }\in \cL_{\Phi_{0}}$, so suppose that
  $\varphi_{\alpha }\in \cL_{\Phi_{0}}$. Then, from
  Example~\ref{examplegamma} necessarily the following two conditions
  are satisfied:
  \begin{enumerate}
  \item $\int_{\abs{x} >1}\log(\abs{x}) \rho(dx) <\infty $,
  \item One of the following conditions holds:
    \begin{enumerate}
    \item $\alpha >-1/2;$
    \item $\alpha =-1/2$, $b=0$ and
      $\int_{\abs{x} \leq 1}\abs{x}^{2}\abs{\log(\abs{x})} \rho(dx) <\infty ;$
    \item $\alpha \in(-1,-1/2)$, $b=0$ and
      $\int_{\abs{x} \leq 1}\abs{x}^{-1/\alpha }\rho(dx) <\infty .$
    \end{enumerate}
  \end{enumerate}

  Since $|\varphi_{\alpha }|\leq c_{1}\phi_{\alpha },$ where
  \begin{equation*}
    \phi_{\alpha }(s)\coloneqq
    \begin{cases}
      s^{\alpha }\indf_{\{0<s\leq 1\}}+e^{-s}\indf_{\{s>1\}}
      & \text{for }-1<\alpha <0;
      \\
      e^{-s}\indf_{\{s\geq 0\}} & \text{for }\alpha \geq 0,
    \end{cases}
  \end{equation*}
  we only need to check that in this case
  $\phi_{\alpha }\in \cL_{\Phi_{1}}$. If $\alpha >0$, from the
  previous example we obtain immediately that
  $\phi_{\alpha }\in \cL_{\Phi_{1}}$. Assume that
  $\alpha \in(-1,0)$. Obviously
  $b^{2}\int_{0}^{\infty }\phi_{\alpha }^{2}(s)ds<\infty$, so it
  suffices to show that
  $\int_{\bR}\int_{0}^{\infty }\left[ |x\phi_{\alpha }(s)|\wedge |x\phi_{\alpha }(s)|^{2}\right] ds\rho (dx)<\infty .$
  From the previous example
  \begin{equation*}
    \int_{\bR}\int_{1}^{\infty }
    \left[ \abs{x\phi_{\alpha }(s)}\wedge \abs{x\phi_{\alpha }(s)}^{2}\right] ds\rho(dx)
    \leq \int_{\bR}\int_{0}^{\infty }\left[ |xe^{-s}|\wedge |xe^{-s}|^{2}\right] ds\rho(dx)<\infty .
  \end{equation*}
  Moreover
  \begin{align*}
    \MoveEqLeft[3]
    \int_{\bR}\int_{0}^{1}\left[ |x\phi_{\alpha }(s)|\wedge |x\phi_{\alpha }(s)|^{2}\right] ds\rho(dx)
    \\
    ={} & \frac{1}{\alpha +1}\int_{|x|>1}|x|\rho(dx)
          +\int_{|x|\leq 1}\int_{0}^{|x|^{-\frac{1}{\alpha }}}|xs^{\alpha }|ds\rho(dx)
    \\
        & +\int_{|x|\leq 1}\int_{|x|^{-\frac{1}{\alpha }}}^{1}|xs^{\alpha}|^{2}ds\rho(dx)
    \\
    ={} & \frac{1}{\alpha +1}\int_{\bR}|x|^{-\frac{1}{\alpha
          }}\wedge|x|\rho(dx) +\frac{1}{2\alpha +1}\int_{|x|\leq
          1}(|x|^{2}-|x|^{-\frac{1}{\alpha }})\rho(dx)<\infty ,
  \end{align*}
  due to the conditions $1.$ and $2$. Thus
  $\phi_{\alpha }\in \cL_{\Phi_{1}}$.

  All above implies that in this case, the integrated process
  $X^{\mu }$ is well defined for any $\alpha >-1$. Now, for
  $\beta >-1$ consider the finite measure
  \begin{equation*}
    \mu(du)\coloneqq \varphi_{\beta }(u)\indf_{\{u\geq 0\}}du,
  \end{equation*}
  and consider the following integrated process
  \begin{align*}
    X_{t}^{\mu }
    \coloneqq {} & \int_{0}^{\infty }X_{t-u}\mu(du)
    \\
    ={} & \int_{-\infty }^{t}\varphi_{\beta }(t-u)X_{u}du,
          \qquad t\in\bR.
  \end{align*}
  From what we have shown above, we see that $X_{t}^{\mu }$ is well
  defined and for each $t\in \bR$ almost surely
  \begin{equation*}
    X_{t}^{\mu }
    =\int_{-\infty }^{t}
    e^{-(t-s)}\int_{s}^{t}(t-u)^{\beta}(u-s)^{\alpha }
    dudL_{s}.
  \end{equation*}
  Since
  \begin{equation*}
    \int_{s}^{t}(t-u)^{\beta }(u-s)^{\alpha }du
    = k_{\alpha ,\beta }(t-s)^{\beta+\alpha +1},
    \qquad t>s,
  \end{equation*}
  where
  $k_{\alpha ,\beta }\coloneqq \int_{0}^{1}x^{\alpha }(1-x)^{\beta }dx<\infty$,
  we have that for $-1<\alpha <0$, $\beta =-\alpha -1$ and $t\in \bR$,
  almost surely
  \begin{equation*}
    X_{t}^{\mu }=k_{\alpha }\int_{-\infty }^{t}e^{-(t-s)}dL_{s},
  \end{equation*}
  with $k_{\alpha }=k_{\alpha ,-\alpha -1}$, i.e.\ $X^{\mu }$ is an OU
  process.  Here we see immediately that
  Theorem~\ref{theoremintegratedprocsd} holds.  Let us remark, that
  the technique of Gamma convolutions has been first used
  in~\cite{BNBV13a} and it has also been applied in~\cite{Sau14}.
\end{example}

\section{\ID field-valued processes}
\label{sec:id-field-valued}

In this section we build L\'{e}vy processes whose realizations are \ID
fields. We propose a way to define stochastic integrals with respect
to such processes and in particular we show that any \SD field can be
expressed as a stochastic integral with respect to an element of this
class of L\'{e}vy processes.

\subsection{\ID field-valued L\'{e}vy processes}

In this part we construct a process which has independent and
stationary increments taking values in the space of random
fields. Thus, in analogy with L\'{e}vy processes in $\bR^{d}$, these
will be called \ID field-valued L\'{e}vy processes.

Let $X=(X_{u})_{u\in U}$ be an \ID field with characteristic triplet
$(\Gamma ,B,\nu)$ and suppose that $\nu$ does not charge zero. For any
$\widehat{u}\in \widehat{U}$ the law of $X_{\widehat{u}}$ belongs to
$\ID(\bR^{\widehat{u}})$ and has characteristic triplet
$(\pi_{\widehat{u}}(\Gamma),B_{\widehat{u}},\nu \circ \pi_{\widehat{u}}^{-1})$
with $B_{\widehat{u}}=(B(u,v))_{u,v\in \widehat{u}}$. Consider
$L_{\widehat{u}}^{X}$ to be a two-sided L\'{e}vy process in
$\bR^{\widehat{u}},$ such that
$L_{\widehat{u}}^{X}(1)\overset{d}{=}X_{\widehat{u}}$. The cumulant
function of $L_{\widehat{u}}^{X}$ is given by
\begin{align}
  \rC\{\theta \ddagger L_{\widehat{u}}^{X}(t)\}
  &{} =|t|\rC\{\theta\ddagger X_{\widehat{u}}\}
    \label{eqn6.1}
  \\
  &{} =i\langle \theta ,\Gamma_{\widehat{u}}^{t}\rangle
    -\tfrac{1}{2}\langle\theta ,B_{\widehat{u}}^{t}\theta \rangle
    +\int_{\bR^{\widehat{u}}}\left[ e^{i\langle \theta ,x\rangle
    } -1-i\langle \tau_{\#\widehat{u}}(x),\theta \rangle \right]
    \nu_{\widehat{u}}^{t}(dx), \notag
\end{align}
where $\theta\in \bR^{\widehat{u}},t\in \bR$, and
\begin{equation*}
  \begin{split}
    \Gamma_{\widehat{u}}^{t} ={} & |t|\pi_{\widehat{u}}(\Gamma);
    \\
    B_{\widehat{u}}^{t} ={} & |t|B_{\widehat{u}};
    \\
    \nu_{\widehat{u}}^{t} ={} & |t|\nu \circ \pi_{\widehat{u}}^{-1}.
  \end{split}
\end{equation*}
Since the system
$(\pi_{\widehat{u}}(\Gamma),B_{\widehat{u}},\nu \circ \pi_{\widehat{u}}^{-1})$
is consistent, we have that for any $t\in \bR$ the system
$(\Gamma_{\widehat{u}}^{t},B_{\widehat{u}}^{t},\nu_{\widehat{u}}^{t})$
is consistent as well, thus from Theorem~\ref{thmmasterlevymeasure}
there exists a unique (in law) \ID field $L^{X}(t)$ with
characteristic triplet $(|t|\Gamma ,|t|B,|t|\nu)$. The process
$L^{X}=(L^{X}(t))_{t\in \bR}$ will be called \emph{\ID field-valued
  L\'{e}vy process}. The motivation of this name is given in the
Proposition~\ref{Proplevyidvalued} below. Before we present such
result, we need to introduce some notation.  For any non empty index
set $U$, let
\begin{equation}
  (\bR^{U})^{\prime }\coloneqq \{y\in \bR^{U}:y(u)=0
  \text{ for all but finitely many }u\in U\}.
  \label{eqn6.3.0}
\end{equation}
Define the pseudo bilinear form
$\langle \nullarg ,\nullarg \rangle_{\bR^{U}}\colon\bR^{U}\times (\bR^{U})^{\prime }\rightarrow \bR$
as
\begin{equation*}
  \langle x,y\rangle_{\bR^{U}}\coloneqq \sum\limits_{u\in U}x(u)y(u).
\end{equation*}
Note that $\langle \nullarg ,\nullarg \rangle_{\bR^{U}}$ is well
defined.  In particular, if $U$ is countable,
$\langle \nullarg ,\nullarg \rangle_{\bR^{U}}$ is the restriction of
the inner product in $l^{2}$ to $(\bR^{U})^{\prime }$.

\begin{proposition}
  \label{Proplevyidvalued}
  The process $L^{X}=(L^{X}(t))_{t\in \bR}$ has independent and
  stationary increments and the process
  \begin{equation}
    \widetilde{L}_{\bm{\cdot} }^{X}(t)
    \coloneqq \pi_{\bm{\cdot} }(t\Gamma) +W_{\bm{\cdot} }(t)
    +\int_{\bR^{U}}\pi_{\bm{\cdot} }(x) \left[ N(dx,ds) -\indf_{\{
        \abs{\pi_{\bm{\cdot} }(x)} \leq 1\} }\widetilde{\nu }(dx,ds) \right]
    ,\quad t\in \bR,  \label{eqn6.3}
  \end{equation}
  is a version of $L^{X}$. Here $W_{\bm{\cdot} }(t)$ is the Gaussian
  process with covariance matrix $\abs{t}(B(u,v))_{u,v\in U}$ and
  $N(dx,ds)$ is a Poisson measure independent of $W_{\bm{\cdot} }(t) $
  with intensity $\widetilde{\nu }(dx,ds) =\nu(dx) ds$ . Moreover, we
  have that
  $\lim_{t\rightarrow s}\langle L^{X}(t) -L^{X}(s) ,y\rangle_{\bR^{U}}$
  exists almost surely and
  \begin{equation}
    \plim_{t\rightarrow s}\langle L^{X}(t) -L^{X}(s),y\rangle_{\bR^{U}}=0,
    \qquad\text{for all }y\in(\bR^{U})^{\prime }.  \label{eqn6.4}
  \end{equation}
\end{proposition}

\begin{proof}
  By construction for any $\widehat{u}\in \widehat{U},$
  $L_{\widehat{u}}^{X}(\nullarg)$ is a L\'{e}vy process in law in
  $\bR^{\widehat{u}}$.  Therefore $L_{\widehat{u}}^{X}(\nullarg)$ has
  independent and stationary increments for any
  $\widehat{u}\in \smash{\widehat{U}}$, so $L^{X}$ does as well.
  Since $\smash{L_{\widehat{u}}^{X}(\nullarg)}$ has independent and
  stationary increments and
  $\smash{\widetilde{L}_{\bm{\cdot}}^{X}(1) \overset{d}{=}X\overset{d}{=}L_{\bm{\cdot} }^{X}(1)}$,
  we have that $L^{X}$ and $\widetilde{L}^{X}$ have the same
  law. Finally, since for any $y\in(\bR^{U})^{\prime }$, there exists
  $\widehat{u}\in \widehat{U}$ such that $y(u)=0$ for all
  $u\in \widehat{u}^{c},$ we deduce
  \begin{align*}
    \langle L^{X}(t)-L^{X}(s),y\rangle_{\bR^{U}}
    ={} & \sum\limits_{u\in\widehat{u}}(L_{u}^{X}(t)-L_{u}^{X}(s))y(u)
    \\
    ={} & \langle
          L_{\widehat{u}}^{X}(t)-L_{\widehat{u}}^{X}(s),y_{\widehat{u}}\rangle_{\bR^{\widehat{u}}},
  \end{align*}
  where $\langle \nullarg ,\nullarg \rangle_{\bR^{\widehat{u}}}$
  denotes the inner product in $\bR^{\widehat{u}}$. This implies
  necessarily that the limit
  $\lim_{t\rightarrow s}\langle L^{X}(t)-L^{X}(s),y\rangle_{\bR^{U}}$
  exists and that \eqref{eqn6.4} holds.
\end{proof}

\begin{remark}
  Observe that in general, the concept of c\`{a}dl\`{a}g paths cannot
  be defined for $L^{X}$. It is because $\bR^{U}$ is not in general
  metric and the topologies that can be defined in such space are not
  tractable. Therefore, the object $\lim_{t\rightarrow s}L^{X}(t)$ may
  not be well defined.  However, if $U\subset \bR^{d}$ and for any
  $t\in \bR$, $L^{X}(t) \in \cL^{2}(\bR) ,$ \eqref{eqn6.4} is
  equivalent to continuity in probability under the $\cL^{2}$-norm. It
  is an open problem to verify that if $L^{X}(t) \in D(U,\bR)$ (the
  Skorohod space) for every $t$ and \eqref{eqn6.4} holds, then $L^{X}$
  is c\`{a}dl\`{a}g under the norm in $D(U,\bR)$.
\end{remark}

\subsection{Integration with respect to \ID field-valued L\'{e}vy
  processes}

In this part, using as a starting point the stochastic integration of
deterministic functions with respect to L\'{e}vy bases on $\bR^{d}$
(see \cite{RajRos89} or Section~\ref{stochintsec} and for the
$\bR^{d}$-valued case see~\cite{BNStel11}), we define the stochastic
integral of an operator from $\bR^{U}$ into itself with respect to an
\ID field-valued L\'{e}vy process.

Let $X=(X_{u})_{u\in U}$ be an \ID field with characteristic triplet
$(\Gamma ,B,\nu)$ and suppose that $\nu$ does not charge
zero. Consider $L^{X}$ to be the L\'{e}vy process \ID field-valued
constructed in the previous subsection. Let
$f\colon\bR \rightarrow \bR$ be a measurable function integrable with
respect to $L_{\widehat{u}}^{X}$, i.e.\
\begin{gather*}
  \int_{\bR}\abs*{
    f(s)\Gamma_{\widehat{u}}
    +\int_{\bR^{ \widehat{u}}}(\tau_{\#\widehat{u}}(f(s)x)-f(s)
    \tau_{\#\widehat{u}}(x))\nu _{\widehat{u}}(dx)} ds<
  \infty ;
  \\
  \int_{\bR^{\widehat{u}}}f^{2}(s)dsB_{\widehat{u}}<\infty ;
  \\
  \int_{\bR}\int_{\bR^{\widehat{u}}}
  (1\wedge \abs{f(s)x}^{2})\nu_{\widehat{u}}(dx)ds<\infty .
\end{gather*}
For the sake of brevity, we introduce the notation
\begin{equation*}
  I_{\widehat{u}}(f\ddagger X)
  \coloneqq \int_{\bR^{\widehat{u}}}f(s)dL_{\widehat{u} }^{X}(s),
  \qquad \widehat{u}\in \widehat{U}.
\end{equation*}
Then $I_{\widehat{u}}(f\ddagger X)$ is \ID with characteristic triplet
$(\Gamma_{ \widehat{u}}^{I(f\ddagger X)},B_{\widehat{u}}^{I(f\ddagger X)},\nu_{\widehat{u}}^{I(f\ddagger X)})$
given by
\begin{align}
  \Gamma_{\widehat{u}}^{I(f\ddagger X)}
  & ={} \int_{\bR}\Bigl[f(s)\Gamma_{\widehat{u}}
    +\int_{\bR^{\widehat{u}}}(\tau_{\#\widehat{u}}(f(s)x)-f(s)
    \tau_{\#\widehat{u}}(x))\nu_{\widehat{u}}(dx)\Bigr]ds,
    \label{eqn6.5} \\
  B_{\widehat{u}}^{I(f\ddagger X)}
  & ={} \int_{\bR^{\widehat{u}}}f^{2}(s)dsB_{ \widehat{u}},  \notag \\
  \nu_{\widehat{u}}^{I(f\ddagger X)}(A)
  & ={} \int_{\bR}\int_{\bR^{\widehat{u}
    }}\indf_{A}(f(s)x)\nu_{\widehat{u}}(dx)ds,
    \qquad A\in\cB(\bR^{\widehat{u}}) .  \notag
\end{align}
This procedure generates a system of characteristic triplets
$(\Gamma_{ \widehat{u}}^{I(f\ddagger X)},B_{\widehat{u}}^{I(f\ddagger X)},\nu_{\widehat{u}}^{I(f\ddagger X)})$
which, as the next proposition shows, is consistent.

\begin{proposition}
  The system of characteristic triplets
  $(\Gamma_{\widehat{u}}^{I(f\ddagger X)},B_{ \widehat{u}}^{I(f\ddagger X)},\nu_{\widehat{u}}^{I(f\ddagger X)})$
  is consistent.
\end{proposition}
\begin{proof}
  Let $\widehat{v},\widehat{u}\in \widehat{U}$ with
  $\widehat{v}\subset \widehat{u}$. From
  Proposition~\ref{propsystemchtrip} and \eqref{eqn6.5}, we only need
  to check that
  \begin{equation*}
    \nu_{\widehat{v}}^{I(f\ddagger X)}
    =\nu_{\widehat{u}}^{I(f\ddagger X)}\circ\pi_{\widehat{v}\widehat{u}}^{-1}
    \qquad\text{on }\cB(\bR^{\widehat{v} }\setminus 0^{\widehat{v}}) .
  \end{equation*}
  From \eqref{eqn6.5}, it follows that for any
  $A\in \cB( \bR^{\widehat{v}}\setminus 0^{\widehat{v}})$
  \begin{align*}
    \nu_{\widehat{u}}^{I(f\ddagger X)}\circ\pi_{\widehat{v}\widehat{u}}^{-1}(A)
    ={} & \int_{\bR}\int_{\bR^{\widehat{u}}}
          \indf_{A}(f(s)\pi_{\widehat{v}\widehat{u}}(x))\nu_{\widehat{u} }(dx)ds
    \\
    ={} & \int_{\bR}\int_{\bR^{\widehat{v}}}\indf_{A}(f(s)x)\nu_{\widehat{v}}(dx)ds
    \\
    ={} & \nu_{\widehat{v}}^{I(f\ddagger X)}(A) ,
  \end{align*}
  which is enough.
\end{proof}

These results mean that we can lift the system of finite-dimensional
triplets
$(\Gamma_{\widehat{u} }^{I(f\ddagger X)},B_{\widehat{u}}^{I(f\ddagger X)},\nu_{\widehat{u}}^{I(f\ddagger X)})$
to a triplet
$(\Gamma^{I(f\ddagger X)},B^{I(f\ddagger X)},\nu^{I(f\ddagger X)})$ of
an \ID field, let's say $I(f\ddagger X)$. In this case, we define the
\emph{stochastic integral of} $f$ \emph{with respect to }$L^{X}$ to be
the \ID field given by
\begin{equation*}
  \int_{\bR}f(s) dL^{X}(s) \coloneqq I(f\ddagger X).
\end{equation*}
Note that, from \eqref{eqn6.5},
\begin{equation*}
  \nu^{I(f\ddagger X)}(A) =\int_{\bR}\int_{\bR^{U}} \indf_{A}(f(s)x)\nu(dx)ds,\qquad A\in \cB(\bR^{U}) .
\end{equation*}
Thus, if $\nu$ does not charge zero, $\nu^{I(f\ddagger X)}$ does not
either. Hence in this case, $\nu^{I(f\ddagger X)}$ is the master
L\'{e}vy measure of $I(f\ddagger X)$. Moreover, a modification of
$I(f\ddagger X)$ can be obtained by the L\'{e}vy-It\^{o}
representation for \ID fields.

\begin{proposition}
  \label{example2sect6}
  Let $X=(X_{u})_{u\in U}$ be an infinitely divisible field with
  characteristic triplet $(\Gamma ,B,\nu)$ such that $\nu$ does not
  charge zero and let $L^{X}$ be the \ID field-valued L\'{e}vy process
  induced by $X$. For a given $\widehat{u}\in \widehat{U}$ denote by
  $\cL_{\Phi_{0}}^{\widehat{u}}$ the Musiela-Orlicz space induced by
  the triplet of $L_{\widehat{u}}^{X}$ (see
  Section~\ref{stochintsec}). Consider $f\colon\bR\rightarrow\bR$ such
  that $f\in \cL_{\Phi_{0}}^{\widehat{u}}$ for every
  $\widehat{u}\in \widehat{U}$. Then the process
  \begin{align*}
    \widetilde{I}_{u}(f\ddagger X)
    \coloneqq{} & \int_{\bR} f(s) ds \pi_{u}(\Gamma)
                  +\int_{\bR} f(s) W_{u}(ds)\\
                & +\int_{\bR}\int_{\bR}
                  f(s)x \left[
                  N_{u}(dx,ds) - \indf_{\{\abs{f(s)x} \leq
                  1\} }\nu_{u}(dx)ds\right] ,
  \end{align*}
  is a modification of $I(f\ddagger X)$. Here $N_{u}$ has compensator
  $\nu_{u}(dx)ds$.
\end{proposition}

\begin{proof}
  By the L\'{e}vy-It\^{o} decomposition, almost surely for $t>s$
  \begin{align*}
    L_{u}^{X}(t) -L_{u}^{X}(s)
    ={} &(t-s) \pi_{u}(\Gamma) +(W_{u}(t) - W_{u}(s)) \\
        & + \int_{s}^{t}\int_{\bR^{U}}
          x \left[ N_{u}(dx,ds)
          - \indf_{\{ \abs{f(s)x} \leq 1\} }\nu_{u}(dx)ds
          \right] .
  \end{align*}
  Therefore, for any $u\in U,$ almost surely
  $I_{u}(F\ddagger X)=\int_{\bR }f(s) dL_{u}^{X}(s) =\widetilde{I}_{u}( F\ddagger X)$,
  as required.
\end{proof}

\begin{remark}
  The procedure above allows to extend the class of integrands to
  linear operators as follows: Let $(f_{u})_{u\in U}$ be a family of
  measurable functions. For every $\widehat{u}\in \widehat{U}$, take
  $F_{\widehat{u}}\colon\bR\rightarrow \bM_{\widehat{u}}(\bR)$ with
  $F_{\widehat{u}}(\nullarg )=\operatorname{diag}(f_{\widehat{u}}(\nullarg))$
  and $\bM_{\widehat{u}}(\bR)$ denotes the set of
  $\#\widehat{u}\times \#\widehat{u}$ matrices with real entries. The
  integral of $F_{\widehat{u}}$ with respect to $L_{\widehat{u}}^{X}$
  (if it exists) can be considered to have a consistent system of
  characteristic triplets
  $(\Gamma_{\widehat{u}}^{I(F_{\widehat{u}}\ddagger X)},B_{\widehat{u}}^{I(F_{\widehat{u}}\ddagger X)},\nu_{\widehat{u}}^{I(F_{\widehat{u}})\ddagger X})$. Moreover,
  the collection $(F_{\widehat{u}})_{\widehat{u}\in \widehat{U}}$ can
  be lifted to an indexed linear operator $F$ from $\bR^{U}$ into
  itself. Therefore, the integral of $F$ with respect to $L^{X},$
  denoted by $I(F\ddagger X)$, is the \ID field with system of
  characteristic triplets given by
  $(\Gamma_{\widehat{u}}^{I(F_{\widehat{u}}\ddagger X)},B_{\widehat{u}}^{I(F_{\widehat{u}}\ddagger X)},\nu_{\widehat{u}}^{I(F_{\widehat{u}}\ddagger X)})$. However,
  it is not clear how to extend this procedure to more general linear
  operators, as the consistency of the system
  $(\Gamma_{\widehat{u}}^{I(F_{\widehat{u}}\ddagger X)},B_{\widehat{u}}^{I(F_{\widehat{u}}\ddagger X)},\nu_{\widehat{u}}^{I(F_{\widehat{u}}\ddagger X)})$
  may fail in the case of non-diagonal operators.
\end{remark}

\subsection{Volterra and OU type field-valued processes and
  selfdecomposability}

Following the steps of the previous subsection, in this part we define
Volterra type \ID field-valued processes, focusing on the OU case.

Let $F\colon\bR\times\bR\rightarrow \bR$ be a measurable function.
Suppose that for each $t\in \bR$, $F(t,\nullarg)$ is integrable with
respect to $L^{X}$ as in the previous subsection. Then the process
$Y(t) \coloneqq I[ F(t,\nullarg)\ddagger X ]$ is well defined as an
\ID field-valued process and we will refer to it as a \emph{Volterra
  type \ID field-valued process}.

A simple yet important example of the functions $F(t,\nullarg)$ is the
one that gives rise to the \emph{Ornstein-Uhlenbeck} \ID field-valued
process.  Namely
\begin{align*}
  F(t,s)=\indf_{(-\infty ,t]}(s)e^{-(t-s)}.
\end{align*}
It is well known that $F$ is $L_{\widehat{u}}^{X}$-integrable if and
only if $\int_{|x|>1}\log(|x|)\nu_{\widehat{u}}(dx)<\infty$ or
equivalently
$\int_{|\pi_{\widehat{u}}(x)|>1}\log (|\pi_{\widehat{u}}(x)|)\nu(dx)<\infty$. In
this setting, the process
\begin{equation}
  \label{eqn6.6}
  Y(t)=\int_{-\infty }^{t}e^{-(t-s)}dL^{X}(s),
\end{equation}
is well defined, provided that
$\int_{|\pi_{\widehat{u}}(x)|>1}\log (|\pi_{\widehat{u}}(x)|)\nu(dx)<\infty$
for any $\widehat{u}\in \widehat{U}$. The field-valued process $Y$ is
\ID and stationary and will be called \emph{field-valued
  Ornstein-Uhlenbeck process}. The following proposition generalizes
the classical result concerning the marginal distributions of OU
processes driven by a L\'{e}vy process.

\begin{proposition}
  Let $X=(X_{u})_{u\in U}$ be an infinitely divisible field with
  characteristic triplet $(\Gamma ,B,\nu)$ such that
  $\smash{\int_{\abs{\pi_{\widehat{u}}(x)} >1}\log(\abs{\pi_{\widehat{u}}(x)}) \nu(dx) <\infty }$
  for every $\widehat{u}\in \widehat{U}$. Take $Y$ to be as in
  \eqref{eqn6.6}. Then, for every $t\in \bR$, the field $Y(t)$ is \SD
  and $Y(t) \overset{d}{=} \int_{0}^{\infty }e^{-s}dL^{X}(s)$.
  Reciprocally, for a given \SD field $Y$, there exists a unique in
  law \ID field-valued L\'{e}vy process $L^{Y},$ such that
  $Y\overset{d}{=}\int_{0}^{\infty }e^{-s}dL^{Y}(s)$.
\end{proposition}

\begin{proof}
  Let $\widehat{u}\in \widehat{U}$. Then
  $Y_{\widehat{u}}(t)\overset{d}{=}\int_{-\infty }^{t}e^{-(t-s)}dL_{\widehat{u}}^{X}(s)\overset{d}{=}\int_{0}^{\infty }e^{-s}dL_{\widehat{u}}^{X}(s)$. It
  is well known that the law of
  $\int_{0}^{\infty }e^{-s}dL_{\widehat{u}}^{X}(s)$ belongs to
  $\SD(\bR^{\widehat{u}})$. Consequently, the \ID field $Y(t)$ is
  selfdecomposable. Reciprocally, let $Y$ be a selfdecomposable field,
  then $\cL(Y_{\widehat{u}})\in \SD(\bR^{\widehat{u}})$, thus there
  exists a unique (in law) L\'{e}vy process $L_{\widehat{u}}^{Y}$ such
  that
  $Y_{\widehat{u}}\overset{d}{=}\int_{0}^{\infty }e^{-s}dL_{\widehat{u}}^{Y}(s)$. Put
  $\widetilde{Y}_{\widehat{u}}(t)\coloneqq \int_{-\infty }^{t}e^{-(t-s)}dL_{\widehat{u}}^{Y}(s)$,
  $t\in \bR$, then $\widetilde{Y}_{\widehat{u}}$ is stationary and its
  marginal distributions are equal to $\cL(Y_{\widehat{u}})$. By the
  Langevin equation
  \begin{equation*}
    L_{\widehat{u}}^{Y}(1)=\widetilde{Y}_{\widehat{u}}(1)-\widetilde{Y}_{\widehat{u}}(1)+\int_{0}^{1}\widetilde{Y}_{\widehat{u}}(s)ds,\qquad \widehat{u}\in \widehat{U},
  \end{equation*}
  i.e.\ $L_{\widehat{u}}^{Y}(1)$ is a functional of
  $\widetilde{Y}_{\widehat{u}}.$ Due to the consistency of the
  characteristic triplets of $\widetilde{Y}_{\widehat{u}}$ we have
  that the triplets of $L_{\widehat{u}}^{Y}(1)$ are consistent as
  well, thus there exists a unique (in law) \ID field-valued L\'{e}vy
  process $L^{Y},$ whose finite-dimensional distributions correspond
  to those of $L_{\widehat{u}}^{Y}$. This concludes the proof.
\end{proof}

\section{Conclusion}
\label{sec:conclusions}
The purpose of this paper has been to study selfdecomposability of
random fields, as defined directly rather than in terms of
finite-dimensional distributions. Applications of the results to
modelling within the framework of Ambit Stochastics will be discussed
elsewhere. The exposition we present is based on the concept of master
L\'evy measures of which we give a thorough discussion, building on
the recent work of \cite{Ros07a,Ros07b,Ros08,Ros13}.
\appendix

\section{Appendix}

In this appendix we present a proof of
Theorem~\ref{thmmasterlevymeasure}.  We want to emphasize that we
construct such a proof based on the remarks given by Jan Rosi\'{n}ski
in~\cite{Ros07a,Ros08,Ros13}.

Let us start with the next lemma.

\begin{lemma}
  Let $(\cX,\cB,\mu)$ be a measure space and $\cL^{1}(\cX,\cB,\mu)$
  the Lebesgue space of real-valued integrable functions. We have that
  $\mu$ is $\sigma $-finite if and only if there is
  $f\in \cL^{1}(\cX,\cB,\mu)$ such that $f$ is strictly positive.
\end{lemma}

\begin{proof}
  The proof is straightforward, thus omitted.
\end{proof}

\begin{lemma}
  \label{proptype1}Let $\nu$ be a measure on $\cB(\bR)^{U}$ satisfying
  \eqref{eqn1.3.7}. Then $\nu $ does not charge zero if and only if
  $\nu$ is $\sigma$-finite and for all $A\in \cB(\bR)^{U}$ there
  exists $U_{A}\subset U$ countable, such that
  \begin{equation}
    \nu(A)=\nu(A\setminus \pi_{U_{A}}^{-1}(0^{U_{A}})).
    \label{eqn1.3.8}
  \end{equation}
\end{lemma}

\begin{proof}
  Assume that Equations~\eqref{eqn1.3.6} and \eqref{eqn1.3.7}
  hold. Then, for any $A\in \cB(\bR)^{U}$,
  $\nu(A\cap \pi_{U_{0}}^{-1}(0^{U_{0}}))=0$, thus
  \begin{equation*}
    \nu(A)=\nu(A\cap \pi_{U_{0}}^{-1}(0^{U_{0}}))+\nu(A\setminus \pi
    _{U_{0}}^{-1}(0^{U_{0}}))=\nu(A\setminus \pi_{U_{0}}^{-1}(0^{U_{0}})),
  \end{equation*}
  proving thus \eqref{eqn1.3.8}. On the other hand, due to
  \eqref{eqn1.3.7}, for any $u\in U_{0}$
  \begin{equation*}
    \int_{A_{0}}1\wedge |\pi_{u}(x)|^{2}\nu(dx)<\infty ,
  \end{equation*}
  with
  $A_{0}\coloneqq \bR^{U}\setminus \pi_{U_{0}}^{-1}(0^{U_{0}})$. Using
  that $1\wedge |\pi_{u}(\nullarg )|^{2}$ is strictly positive on
  $A_{0}$ and the previous lemma, we have that $\nu$ restricted to
  $A_{0}$ is $\sigma$-finite, i.e.\ there exists
  $\{S_{n}^{\prime }\}_{n\geq 1}$, such that
  $S_{n}^{\prime }\uparrow \bR^{U}$ and
  $\nu(S_{n}^{\prime }\cap A_{0})<\infty$ for all $n\in \bN$. Putting
  $S_{n}=(S_{n}^{\prime }\cap A_{0})\cup \pi_{U_{0}}^{-1}(0^{U_{0}})$,
  we see that $S_{n}\uparrow \bR^{U}$ and thanks to \eqref{eqn1.3.6}
  \begin{equation*}
    \nu(S_{n})\leq \nu(S_{n}^{\prime }\cap A_{0})<\infty ,
  \end{equation*}
  i.e.\ $\nu$ is $\sigma$-finite. Conversely, assume that $\nu$ is
  $\sigma$-finite and \eqref{eqn1.3.8} holds, then without loss of
  generality we may and do assume that $\nu$ is finite. Thanks to
  \eqref{eqn1.3.8}, we have that there exist $U_{0}\subset U$
  countable, such that
  \begin{equation*}
    \nu(\bR^{U}\setminus \pi_{U_{0}}^{-1}(0^{U_{0}}))=\nu(\bR^{U}),
  \end{equation*}
  which implies \eqref{eqn1}.
\end{proof}

\begin{lemma}
  \label{lemma2}
  The collection
  $\cB_{0}\coloneqq \bigcup_{\widehat{u}\in \widehat{U}}\pi_{\widehat{u}}^{-1}[ \cB(\bR^{ \widehat{u}}\setminus 0^{\widehat{u}}) ]$
  is a ring for which $\cB(\bR)^{U}=\sigma( \cB_{0})$. Let $\nu$ and
  $\widetilde{\nu }$ be two $\sigma $-finite measures defined on
  $\cB(\bR)^{U}$ which coincide in $\cB_{0}$. If $\nu$ and
  $\widetilde{\nu }$ do not charge zero, then
  $\nu \equiv \widetilde{\nu }$.
\end{lemma}

\begin{proof}
  Obviously $\cB_{0}$ is a ring and
  $\cB_{0}\subset \cB(\bR)^{U}$. Therefore,
  $\cB(\bR)^{U}=\sigma(\cB_{0})$ if
  $\pi_{_{\widehat{u}}}^{-1}(A)\in \sigma(\cB_{0})$ for any
  $A\in \cB(\bR^{\widehat{u}})$ and $\widehat{u}\in \widehat{U}$. Fix
  $\widehat{u}\in \widehat{U}$ and take an arbitrary
  $A\in \cB(\bR^{\widehat{u}})$. Now, if $0^{\widehat{u}}\notin A,$
  $\pi_{_{\widehat{u}}}^{-1}(A)\in \pi_{\widehat{u}}^{-1}[\cB(\bR^{\widehat{u}}\setminus 0^{\widehat{u}})]$,
  i.e.\ $\pi_{_{\widehat{u}}}^{-1}(A)\in \sigma (\cB_{0})$. In
  counterpart, if $0^{\widehat{u}}\in A$,
  $0^{\widehat{u}}\notin A^{c}$, thus as before
  $\pi_{_{\widehat{u}}}^{-1}(A^{c})\in \sigma(\cB_{0})$, which implies
  necessary that
  $\pi_{_{\widehat{u}}}^{-1}(A)=[\pi_{_{\widehat{u}}}^{-1}(A^{c})]^{c}\in \sigma(\cB_{0})$.

  On the other hand, let $\nu$ and $\widetilde{\nu }$ be two
  $\sigma $-finite measures coinciding on $\cB_{0}$. By the
  $\sigma $-finiteness, we may and do assume that $\nu$ and
  $\widetilde{\nu }$ are finite measures.  Invoking the Monotone Class
  Theorem, we deduce that $\nu$ and $\widetilde{ \nu }$ coincide on
  $\cS\cR(\cB_{0})$ the $\sigma$-ring generated by $\cB_{0}$. Suppose
  now that $\nu$ and $\widetilde{\nu }$ do not charge zero. Then,
  there are $U_{0}^{1},U_{0}^{2}\subset U$ countable such that
  $\nu( \bR^{U}\setminus \pi_{U_{0}^{1}}^{-1}(0^{U_{0}^{1}})) = \widetilde{\nu }(\bR^{U}\setminus \pi_{U_{0}^{2}}^{-1}( 0^{U_{0}^{2}})) =0$. Putting
  $U_{0}=U_{0}^{1}\cup U_{0}^{2}$, we get
  $\nu(\pi_{U_{0}}^{-1}(0^{U_{0}})) = \widetilde{\nu }( \pi_{U_{0}}^{-1}(0^{U_{0}})) =0$.
  This means that in order to complete the proof we only need to check
  that
  $\nu( \bR^{U}\setminus \pi_{U_{0}}^{-1}(0^{U_{0}})) =\widetilde{\nu }(\bR^{U}\setminus \pi_{U_{0}}^{-1}( 0^{U_{0}}))$,
  because in this case $\nu$ and $\widetilde{\nu }$ will coincide on
  $\cB_{0}\cup \{ \bR^{U}\}$, which implies, by the first part of the
  lemma and once you apply the Monotone Class Theorem, that $\nu$ and
  $\widetilde{\nu }$ coincide on
  $\cS\cR(\cB_{0}\cup \{ \bR^{U}\} ) =\sigma(\cB_{0}) =\cB(\bR)^{U}$.

  Let us verify that
  $\nu(\bR^{U}\setminus \pi_{U_{0}}^{-1}( 0^{U_{0}})) =\widetilde{\nu }(\bR^{U}\setminus \pi_{U_{0}}^{-1}(0^{U_{0}}))$. To
  do this, we show that
  $\bR^{U}\setminus \pi_{U_{0}}^{-1}(0^{U_{0}}) \in \cS\cR(\cB_{0})$.
  Assume that $U_{0}$ is finite. In view of
  $\bR^{U}\setminus \pi_{U_{0}}^{-1}( 0^{U_{0}}) =\pi_{U_{0}}^{-1}(\bR^{U_{0}}\setminus 0^{U_{0}}) $
  we see that
  $\bR^{U}\setminus \pi_{U_{0}}^{-1}(0^{U_{0}}) \in \cB_{0}\subset \cS\cR(\cB_{0}) $. Suppose
  now that $U_{0}$ has infinitely many elements, lets say
  $(u_{n})_{n\in \bN }\subset U$. Define
  $U_{0}^{n}\coloneqq( u_{i})_{i=1}^{n},$ then
  $U_{0}^{n}\in \widehat{U}$, $U_{0}^{n}\uparrow U_{0}$ and
  $\pi_{U_{0}^{n}}^{-1}(0^{U_{0}^{n}}) \downarrow \pi_{U_{0}}^{-1}(0^{U_{0}})$. Consequently
  $\pi_{U_{0}^{n}}^{-1}(\bR^{U_{0}^{n}}\setminus 0^{U_{0}^{n}}) \in \cB_{0}\subset \cS\cR( \cB_{0})$
  for all $n\in \bN$ and
  \begin{equation*}
    \bR^{U}\setminus \pi_{U_{0}}^{-1}(0^{U_{0}})
    =\bigcup\limits_{n\geq 1}\pi_{U_{0}^{n}}^{-1}(\bR
    ^{U_{0}^{n}}\setminus 0^{U_{0}^{n}}) \in \cS\cR(\cB_{0}) ,
  \end{equation*}
  which completes the proof.
\end{proof}

\begin{lemma}
  \label{lem:1}
  Suppose that for all $n\in \bN$ we have a non-empty compact set
  $C_{n}\subseteq \bR^{n}$ such that
  $(x_{1},\ldots ,x_{n+1})\in C_{n+1}$ implies that
  $(x_{1},\ldots ,x_{n})\in C_{n}$. Then there is
  $y_{\infty }=(y_{1},y_{2},\ldots )\in \bR^{\bN}$ such that
  $(y_{1},\ldots ,y_{n})\in C_{n}$ for all $n\in \bN$.
\end{lemma}

\begin{proof}
  By the continuity of projections, for any $m,n\in \bN$, $m>n$ the
  set $\pi_{mn}(C_{m})$ is non-empty and compact. It is easy to check
  that
  $C_{n}=\pi_{nn}(C_{n})\supset \pi_{n+1\,n}(C_{n+1})\supset \ldots$,
  for any $n\in \bN$. Since an intersection of a decreasing sequence
  of non-empty compact sets is itself non-empty, there is an $y_{1}$
  such that $y_{1}\in \bigcap_{n\geq 1}\pi_{n1}(C_{n})$.

  Consider
  $C_{n}(y_{1})=\pi_{n-1}[ \pi_{n1}^{-1}(\{y_{1}\})\cap C_{n} ]$ for
  $n\geq2$.  By construction, the family $\{C_{n}(y_{1})\}_{n\geq 2}$
  satisfies the assumptions of this Lemma. Repeating the argument that
  lead to the choice of $y_{1}$, we can find $y_{2}$ such that for any
  $n>2$, $(y_{1},y_{2})\in \pi_{n2}(C_{n})$. Now, by induction, we
  obtain $y=(y_{1},y_{2},\ldots)\in \bR^{\bN}$ such that
  $\pi_{n}(y)\in C_{n}$ for each $n\in \bN$, as required.
\end{proof}

Now we are ready to show a proof of
Theorem~\ref{thmmasterlevymeasure}:

\begin{proof} [Proof of Theorem~\ref{thmmasterlevymeasure}]
  Firstly we prove the uniqueness. From
  Proposition~\ref{propsystemchtrip}, the functions $\Gamma$ and $B$
  are unique, so we only need to check that if there is a measure
  $\widetilde{\nu }$ that does not charge zero and satisfies
  \eqref{eqn1.3.4}, then $\nu \equiv \widetilde{\nu }$. If $\nu$ and
  $\widetilde{\nu }$ are two measures satisfying \eqref{eqn1.3.4} that
  do not charge zero, then \eqref{eqn1.3.7} holds for $\nu$ and
  $\widetilde{\nu }$ and they coincide on $\cB_{0}$. By
  Lemma~\ref{proptype1} $\nu$ and $\widetilde{\nu }$ are
  $\sigma $-finite. The uniqueness follows from Lemma~\ref{lemma2}.

  Now we proceed to prove the existence. We divide the proof in four
  steps to make it easier to read. As a starting point, we define a
  measure on $\cB_{0}$, and then we extend it to the $\sigma $-algebra
  generated by $\cB_{0}$, which in virtue of Lemma~\ref{lemma2}
  coincides with $\cB(\bR)^{U}$.

  \paragraph{Step 1: Defining the pre-measure}

  Recall that if $A_{0}\in \cB_{0}$, then there is a
  $\widehat{u}\in \widehat{U}$ and $A\in \cB( \bR^{\widehat{u}})$,
  such that
  $A_{0}=\pi_{\widehat{u}}^{-1}(A\setminus 0^{\widehat{u} }) .$ Let us
  define the set function
  $\nu_{0}\colon\cB_{0}\rightarrow [ 0,\infty ]$, by
  \begin{equation}
    \nu_{0}(A_{0}) =\nu_{\widehat{u}}(A) ,
    \qquad  A_{0}\in \cB_{0},
    \label{eqn1.3.9}
  \end{equation}
  provided that
  $A_{0}=\pi_{\widehat{u}}^{-1}(A\setminus 0^{\widehat{u }})$. Here
  $\nu_{\widehat{u}}$ is the L\'{e}vy measure of $X_{ \widehat{u}}$.
  We claim that $\nu$ is well defined. Indeed, suppose that there are
  $\widehat{u},\widehat{v}\in \widehat{U}$ and
  $A^{1}\in \cB(\bR^{\widehat{u}}) $,
  $A^{2}\in \cB(\bR^{\widehat{v}})$ such that
  $A_{0}=\pi_{\widehat{u} }^{-1}(A^{1}\setminus 0^{\widehat{u}}) =\pi_{\widehat{v} }^{-1}(A^{2}\setminus 0^{\widehat{v}})$. Then,
  $A_{0}=\pi_{ \widehat{u}}^{-1}(A^{1}\setminus 0^{\widehat{u}}) =\pi_{ \widehat{w}}^{-1}[ \pi_{\widehat{w}\widehat{v}}^{-1}( A^{2}\setminus 0^{\widehat{v}}) ]$
  where $\widehat{w}=\widehat{ u}\cup \widehat{v}\in \widehat{U}.$
  Thus, by Proposition~\ref{propsystemchtrip}
  \begin{align*}
    \nu_{\widehat{u}}(A^{1})
    ={} & \nu_{0}(A_{0})                                                                           \\
    ={} & \nu_{\widehat{w}}(\pi_{\widehat{w}\widehat{v}}^{-1}(A^{2}\setminus 0^{\widehat{v}})) \\
    ={} & \nu_{\widehat{v}}(A^{2}) ,
  \end{align*}
  where we used that
  $\pi_{\widehat{w}\widehat{v}}^{-1}( A^{2}\setminus 0^{\widehat{v}}) \in \cB( \bR^{\widehat{w} }\setminus 0^{\widehat{w}})$. This
  implies that $\nu$ is well defined.

  \paragraph{Step 2: $\nu_{0}$ is finitely additive}

  In this step we show that $\nu_{0}$ is finitely additive. Let
  $A_{0},B_{0}\in \cB_{0}$ with $A_{0}\cap B_{0}=\emptyset $. There
  exist $\widehat{u},\widehat{v}\in \widehat{U},$
  $A\in \cB(\bR^{\widehat{u}})$ and $B\in \cB( \bR^{ \widehat{v}})$
  such that $A_{0}=\pi_{\widehat{u}}^{-1}(A\setminus 0^{\widehat{u}})$
  and $B_{0}=\pi_{\widehat{v} }^{-1}(B\setminus 0^{\widehat{v}}) .$
  Put
  $\widehat{w}=\widehat{ u}\cup \widehat{v}\in \widehat{U}$. Again, by
  Proposition~\ref{propsystemchtrip}
  \begin{align*}
    \nu_{0}(A_{0}\cup B_{0})
    ={} & \nu_{0}\{ \pi_{\widehat{w} }^{-1}[ \pi_{\widehat{w}\widehat{u}}^{-1}(A\setminus 0^{ \widehat{u}}) \cup \pi_{\widehat{w}\widehat{v}}^{-1}(B\setminus 0^{\widehat{v}}) ] \} \\
    ={} & \nu_{\widehat{w}}[ \pi_{\widehat{w}\widehat{u}}^{-1}(A\setminus 0^{\widehat{u}}) \cup \pi_{\widehat{w}\widehat{v} }^{-1}(B\setminus 0^{\widehat{v}}) ]                     \\
    ={} & \nu_{\widehat{w}}[ \pi_{\widehat{w}\widehat{u}}^{-1}(A\setminus 0^{\widehat{u}}) ] +\nu_{\widehat{w}}[ \pi_{ \widehat{w}\widehat{v}}^{-1}(B\setminus 0^{\widehat{v}}) ]    \\
    ={} & \nu_{0}(A_{0}) +\nu_{0}(B_{0}) ,
  \end{align*}
  thanks to
  $\pi_{\widehat{w}\widehat{u}}^{-1}(A\setminus 0^{\widehat{u }}) \cap \pi_{\widehat{w}\widehat{v}}^{-1}(B\setminus 0^{ \widehat{v}}) =\emptyset$. This
  also implies that $\nu_{0}( \emptyset) =0$.

  \paragraph{Step 3: Continuity at the empty set}

  Recall that a set function $\mu$ defined on $\cR$ a ring of sets
  which is finitely additive is $\sigma$-additive if and only if it is
  continuous at the empty set, i.e.\ if $A_{n}\downarrow \emptyset$
  with $A_{n}\in \cR$ then $\mu(A_{n}) \rightarrow 0$ or equivalently
  if $(A_{n})_{n\geq 1}$ is a decreasing sequence on $\cR$ with
  $\inf_{n\geq 1}\mu( A_{n}) >0$ then
  $\bigcap_{n\geq 1}A_{n}\neq \emptyset$.

  Now we prove that $\nu_{0}$ as in \eqref{eqn1.3.9} is continuous at
  the empty set. Let $(A_{n}^{0})_{n\geq 1}\subset\nobreak \cB_{0}$,
  then there are $\widehat{u}_{n}\in \widehat{U}$ and
  $A_{n}\in \cB(\bR^{\widehat{u}_{n}})$, such that
  $A_{n}^{0}=\pi_{ \widehat{u}_{n}}^{-1}(A_{n}\setminus 0^{\widehat{u}_{n}})$
  for any $n\in \bN$. Without loss of generality we may assume that
  $\widehat{u}_{n}\subset \widehat{u}_{n+1},$ otherwise put
  $\widehat{u}_{n}^{\prime }=\bigcup_{i=}^{n}\widehat{u}_{n}$ and use
  that
  $\pi_{\widehat{u}_{n}}^{-1}(A_{n}\setminus 0^{\widehat{u}_{n}}) =\pi_{\widehat{u }_{n}^{\prime }}^{-1}(\pi_{\widehat{u}_{n}^{\prime }\widehat{u}_{n}}^{-1}(A_{n}\setminus 0^{\widehat{u}_{n}}))$.

  Consider $(A_{n}^{0})_{n\geq 1}$ be a decreasing sequence with
  $\inf_{n\geq 1}\nu_{0}(A_{n}^{0})>0$. We want to show that
  $\bigcap_{n\in \bN}A_{n}^{0}\not=\emptyset$. The condition
  with the infimum is equivalent to saying that there is an $\epsilon
  >0$ such that for all $n\in \bN$,
  $\nu_{\widehat{u}_{n}}(A_{n})>\epsilon$. Considering that
  $\nu_{\widehat{u}_{n}}$ is a L\'{e}vy measure for any
  $\widehat{u}_{n}\in \widehat{U}$, we have that there is a compact
  set $K_{n}\subseteq A_{n}\setminus 0^{\widehat{u}_{n}}\subset
  \bR^{\widehat{u}_{n}}\setminus 0^{\widehat{u}_{n}}$ such that
  \begin{equation}
    \nu_{\widehat{u}_{n}}[(A_{n}\setminus 0^{\widehat{u}_{n}})\setminus
    K_{n}]=\nu_{\widehat{u}_{n}}(A_{n}\setminus K_{n})<\frac{\epsilon }{2^{n+1}},\qquad n\in \bN.  \label{eq:3}
  \end{equation}
  Let
  $C_{n}=\bigcap_{k=1}^{n}\pi_{\widehat{u}_{n}\widehat{u}_{k}}^{-1}(K_{k}) $. We
  see that $C_{n}$ is compact on $\bR^{\widehat{u}_{n}}$ with
  $C_{n}\subset K_{n}\subset A_{n}\setminus 0^{\widehat{u}_{n}}$. Further,
  $C_{n}$ is non-empty for all $n\in \bN$. Indeed, from (\ref{eq:3})
  and the fact that $0^{\widehat{u}_{n}}\notin C_{n}$, we have
  \begin{equation*}
    \epsilon
    < \nu_{\widehat{u}_{n}}(A_{n}\setminus C_{n})
    + \nu_{\widehat{u}_{n}}(C_{n})
    < \nu_{\widehat{u}_{n}}(C_{n})
    + \tfrac{\epsilon }{2},
  \end{equation*}
  or in other words
  $\nu_{\widehat{u}_{n}}(C_{n})>\frac{\epsilon }{2}$. By construction
  $\bigcap_{n\geq 1}\pi_{\widehat{u}_{n}}^{-1}(C_{n})\subset \bigcap_{n\geq 1}A_{n}^{0}$,
  meaning that in order to show that
  $\bigcap_{n\in \bN}A_{n}^{0}\not=\emptyset$ we only need to check
  that
  $\bigcap_{n\geq 1}\pi_{\widehat{u}_{n}}^{-1}(C_{n})\neq \emptyset $. Further,
  defining
  $\widehat{u}_{\infty }\coloneqq \bigcup_{n\geq 1}\widehat{u}_{n}$
  and putting
  $C_{n}^{_{\infty }}\coloneqq \pi_{\widehat{u}_{\infty }\widehat{u}_{n}}^{-1}(C_{n})$
  with $n\in \bN$, we get
  $\bigcap_{n\geq 1}\pi_{\widehat{u}_{n}}^{-1}(C_{n})=\pi_{\widehat{u}_{\infty }}^{-1}(\bigcap_{n\geq 1}C_{n}^{_{\infty }})$. Therefore,
  it suffices to prove that $\bigcap_{n\geq 1}C_{n}^{_{\infty }}$ is
  non-empty. Note that if $\widehat{u}_{\infty }$ is a finite set,
  then $(C_{n}^{_{\infty }})_{n\geq 1}$ is a collection of non-empty
  compact sets on $\bR^{\widehat{u}_{\infty }}$, implying trivially
  that $\bigcap_{n\geq 1}C_{n}^{_{\infty }}\neq \emptyset ,$ so we
  only consider the case when $\widehat{u}_{\infty }$ has infinitely
  many elements. Since $\widehat{u}_{n}\subset \widehat{u}_{n+1}$, we
  can assume that there is $(u_{n})_{n\in \bN}\subset U$ such that
  $\widehat{u}_{n}=(u_{i})_{i=1}^{n}$. Note that in this case
  $\pi_{\widehat{u}_{n+1}\widehat{u}_{n}}(C_{n+1})\subset C_{n}$ due
  to
  $\pi_{\widehat{u}_{n+1}\widehat{u}_{k}}^{-1}(K_{k})=\pi_{\widehat{u}_{n+1}\widehat{u}_{n}}^{-1}(\pi_{\widehat{u}_{n}\widehat{u}_{k}}^{-1}(K_{k}))$. Hence,
  from Lemma~\ref{lem:1} there is $x\in \bR^{\widehat{u}_{\infty }}$
  such that $\pi_{\widehat{u}_{\infty }\widehat{u}_{n}}(x)\in C_{n}$
  for all $n\in \bN$ concluding thus that
  $\bigcap_{n\geq 1}C_{n}^{_{\infty }}\neq \emptyset$.

  \paragraph{Step 4: Extending $\nu_{0}$}

  At this point we have so far that $\nu_{0}$ is a $\sigma$-additive
  measure on the ring $\cB_{0}$. By the Carath\'eodory Extension
  Theorem, it follows that there is an extension of $\nu_{0}$, lets
  say $\nu ,$ to $\sigma(\cB_{0}) =\cB(\bR)^{U}$, such that
  $\nu \mid_{\cB_{0}}=\nu_{0}$. This step concludes the proof.
\end{proof}

\section*{Acknowledgments}
Financial support from the Center for Research in the Econometric
Analysis of Time Series (grant DNRF78) funded by the Danish National
Research Foundation is gratefully acknowledged.

\nocite{*} \bibliographystyle{plainnat} \bibliography{sdfields-arxiv}

\end{document}